\documentclass[a4paper,11pt, reqno]{amsart}

\usepackage{amsmath}
\usepackage{amssymb}
\usepackage{mathrsfs}
\usepackage{ifthen}
\usepackage{graphicx}
\usepackage[T1]{fontenc} 
\usepackage{lipsum}
\usepackage[a4paper,margin=0.75in]{geometry}

\usepackage{cite}

\def\switchlinenumbers{\@ifstar
	{\let\makeLineNumberOdd\makeLineNumberRight
		\let\makeLineNumberEven\makeLineNumberLeft}%
	{\let\makeLineNumberOdd\makeLineNumberLeft
		\let\makeLineNumberEven\makeLineNumberRight}%
}

\def\setmakelinenumbers#1{\@ifstar
	{\let\makeLineNumberRunning#1%
		\let\makeLineNumberOdd#1%
		\let\makeLineNumberEven#1}%
	{\ifx\c@linenumber\c@runninglinenumber
		\let\makeLineNumberRunning#1%
		\else
		\let\makeLineNumberOdd#1%
		\let\makeLineNumberEven#1%
		\fi}%
}

\nonstopmode \numberwithin{equation}{section}
\newtheorem*{theorem*}{Theorem}

\newtheorem{thm}{Theorem}[section]
\newtheorem{cor}{Corollary}[section]
\newtheorem{lem}{Lemma}[section]

\theoremstyle{definition}

\newtheorem{example}{Example}[section]
\newtheorem{qsn}{Question} [section]
\newtheorem{prob}[equation]{Problem}
\newtheorem{rem}{Remark}[section]

\newtheorem{innercustomthm}{Theorem}
\newenvironment{customthm}[1]
{\renewcommand\theinnercustomthm{#1}\innercustomthm}
{\endinnercustomthm}

\newcounter{minutes}
\divide\time by 60
\newcounter{hours}
\multiply\time by 60
\addtocounter{minutes}{-\time}

\newcounter {own}
\def\theown {\thesection       .\arabic{own}}

\newenvironment{pf}[1][]{%
	\vskip 3mm
	\noindent
	\ifthenelse{\equal{#1}{}}%
	{{\slshape {\bf Proof}. }}%
	{{\slshape #1.} }%
}%
{\qed\bigskip}

\newcounter{alphabet}

\def\be{\begin{equation}}
	\def\ee{\end{equation}}

\newcommand{\bee}{\begin{enumerate}}
	\newcommand{\eee}{\end{enumerate}}

\newcommand{\blem}{\begin{lem}}
	\newcommand{\elem}{\end{lem}}
\newcommand{\bthm}{\begin{thm}}
	\newcommand{\ethm}{\end{thm}}
\newcommand{\bcor}{\begin{cor}}
	\newcommand{\ecor}{\end{cor}}
\newcommand{\beg}{\begin{examp}}
	\newcommand{\eeg}{\end{examp}}
\newcommand{\begs}{\begin{examples}}
	\newcommand{\eegs}{\end{examples}}
\newcommand{\bdefe}{\begin{defin}}
	\newcommand{\edefe}{\end{defin}}
\newcommand{\bprob}{\begin{prob}}
	\newcommand{\eprob}{\end{prob}}
\newcommand{\bei}{\begin{itemize}}
	\newcommand{\eei}{\end{itemize}}
\newcommand{\real}{{\operatorname{Re}\,}}

\newcommand{\norm}[1]{\left\lVert#1\right\rVert}

\newcommand{\comment}[1]{}

\subjclass[{AMS} Subject Classification:]{Primary 32A05, 31C10, 46B07;  Secondary 32Q02, 46E40}

\begin{document}
	
	\title[]{On the second Bohr radius for vector valued pluriharmonic functions}

	\author{Himadri Halder}
	\address{Himadri Halder,
		Department of Mathematics,
		Indian Institute of Science, Bangalore-560012, India}
	\email{himadrihalder119@gmail.com, himadrih@iisc.ac.in}
	

	
	
	\begin{abstract}
		In this paper, we introduce the notion of the second Bohr radius for vector valued pluriharmonic functions on complete Reinhardt domains in $\mathbb{C}^n$. This investigation is motivated by the work of
		Lev Aizenberg [Proc. Amer. Math. Soc. 128 (2000), 1147  –1155], where the corresponding problem was studied for complex valued holomorphic functions. We show that the second Bohr radius constant for pluriharmonic functions is strictly positive under suitable condition. In addition, we obtain its asymptotic behavior in both the finite- and infinite-dimensional settings using invariants from local Banach space theory. Asymptotic estimates for this constant are obtained on both convex and non-convex complete Reinhardt domains. Our results also apply to a broad class of Banach sequence spaces, including symmetric and convex Banach spaces. The framework developed here also includes the second Bohr radius problem for vector valued holomorphic functions. As an application of our results, we derive several consequences that extend known results in the scalar valued setting as well as existing results in the literature.
	\end{abstract}

	\maketitle
	\pagestyle{myheadings}
	\markboth{Himadri Halder}{Second Bohr radius for vector valued pluriharmonic functions}
	\tableofcontents
	\section{Introduction}\label{section-1}
	The classical theorem of Harald Bohr \cite{Bohr-1914}, proved over a century ago, has developed into a broad and active area of research in complex and functional analysis, now referred to as Bohr’s theorem and the Bohr radius problem. Bohr's original result concerns holomorphic functions of one complex variable. Interest in this theme was revived in the $1990$s following its successful application in Banach algebras (see \cite{Dixon & BLMS & 1995}), as well as subsequent generalizations to holomorphic functions of several complex variables (see \cite{boas-1997,aizn-2000a}) and to more abstract settings in functional analysis (see \cite{aizn-2000b}), especially those related to local Banach space theory (see \cite{defant-2003}).
\vspace{1mm}

	In higher dimensions, several approaches have been proposed to generalize Bohr’s theorem. A significant breakthrough was achieved by Boas and Khavinson (see \cite{boas-1997}), who introduced the notion of the Bohr radius, now referred to as the first Bohr radius, for complete Reinhardt domains in 
	$\mathbb{C}^n$. Their work opened a new direction in the function theory of several complex variables, revealing deep connections with functional analysis, including fundamental aspects of both the local and global theory of Banach spaces. Independently, Aizenberg \cite{aizn-2000a} introduced another multidimensional analogue, known as the second Bohr radius, for complete Reinhardt domains in $\mathbb{C}^n$. 
	\par		
	Among the various aspects of Bohr's theorem, a particularly intriguing direction emerges when the problem is studied for vector valued holomorphic functions on complete Reinhardt domains through the lens of local Banach space theory. This viewpoint naturally builds a bridge between the function theory of several complex variables and functional analysis. Along these lines, Blasco \cite{Blasco-OTAA-2010,Blasco-Collect-2017} was the first to initiate the study of the Bohr phenomenon for Banach space valued holomorphic functions defined on the unit disc. Subsequently, Defant, Maestre, and Schwarting \cite{defant-2011} extended this investigation to vector valued holomorphic functions by studying first Bohr radius on the unit polydisc. A major breakthrough was achieved by Defant, Garc\'{i}a, and Maestre \cite{defant-2003}, who studied the first Bohr radius problem for complex valued holomorphic functions on the unit ball of finite-dimensional Banach spaces with a canonical basis that is normalized and 1-unconditional. 
	\vspace{1mm}
	
	More recently, the present author \cite{Himadri-local-Banach-1} carried out a systematic study of the first Bohr radius problem for vector valued holomorphic and pluriharmonic functions defined on bounded complete Reinhardt domains in $\mathbb{C}^n$, analyzing the asymptotic behavior of the Bohr radius in both finite- and infinite-dimensional settings. The framework developed in this article extends the classical Minkowski-space setting to a substantially broader class of symmetric Banach sequence spaces.
	\vspace{1mm}
	
	To the best of our knowledge, the second Bohr radius problem for pluriharmonic functions remains unexplored. In this work, we address this problem for vector valued pluriharmonic functions on complete Reinhardt domains via local Banach space theory. We derive asymptotic estimates of this constant in both finite and infinite dimensions. As an application, we obtain results for Banach sequence spaces, demonstrating that our framework applies to a broader class of domains in $\mathbb{C}^n$. 
	\vspace{1mm}
	
	In order to address this problem, we first recall the definitions of the first Bohr radius for vector valued holomorphic and pluriharmonic functions introduced in \cite{Himadri-local-Banach-1} (see also \cite{defant-2011}). Let $\Omega \subset \mathbb{C}^n$ be a complete Reinhardt domain and let $n \in \mathbb{N}$. Let $\mathcal{H}(\Omega,X)$ denote the set of all bounded $X$ valued holomorphic functions on $\Omega$. Let $U:X\rightarrow Y$ be a non-null bounded linear operator between complex Banach spaces $X$ and $Y$, with $\norm{U} \leq \lambda$. For $1 \leq p < \infty$, the first $\lambda$-powered Bohr radius of $U$ with respect to the class $\mathcal{H}(\Omega,X)$, denoted by $K_{\lambda}(\Omega, p,U)$, is defined as the supremum of all $r \in [0,1]$ such that for every $f \in \mathcal{H}(\Omega,X)$ of the form $f(z)=\sum_{\alpha}a_{\alpha}z^{\alpha}$ one has 
	\begin{equation} \label{e-1.2} \sup_{z \in r\Omega}\, \sum_{\alpha} \norm{U(a_{\alpha})z^\alpha}^p_{Y} \leq \lambda^p\,\norm{f}^p_{\Omega,X}, \end{equation} 
	where $\norm{f}_{\Omega,X}:=\sup_{z \in \Omega}\norm{f(z)}_{X}$. We now fix some notation. We write $K(\Omega, p,U):=K_{1}(\Omega, p,U)$. When $U=I:X\rightarrow X$ is the identity operator, we set $K_{\lambda}(\Omega, p,X):=K_{\lambda}(\Omega, p,I)$ and $K(\Omega, p,X):=K_{1}(\Omega, p,X)$. In the scalar-valued case, we denote $K_{\lambda}(\Omega, p):=K_{\lambda}(\Omega, p,\mathbb{C})$ and $K(\Omega,p):=K_1(\Omega,p)$. With these notations, the classical Bohr's theorem asserts that $K(\mathbb{D},1)=1/3$ (see \cite{Bohr-1914}). The constant $K(\Omega,1)$ was first investigated by Boas and Khavinson (see \cite{boas-1997,boas-2000}). We return to a detailed discussion of this constant in Section $3$.
	\vspace{1mm}
	
	 Recall that, on a simply connected complete Reinhardt domain $\Omega$, a continuous function $f:\Omega \rightarrow \mathcal{B}(\mathcal{H})$ is pluriharmonic if and only if it admits a decomposition of the form
	 \begin{equation} \label{e-1.3-a}
	 	f(z)=\sum_{m=0}^{\infty} \sum_{|\alpha|=m} a_{\alpha}\, z^{\alpha} + \sum_{m=1}^{\infty} \sum_{|\alpha|=m} b^{*}_{\alpha}\, \bar{z}^{\alpha},
	 \end{equation}
	 where $h(z)=\sum_{m=0}^{\infty} \sum_{|\alpha|=m} a_{\alpha}\, z^{\alpha}$ and $g(z)=\sum_{m=1}^{\infty} \sum_{|\alpha|=m} b_{\alpha}\, z^{\alpha}$ are $\mathcal{B}(\mathcal{H})$-valued holomorphic functions on $\Omega$. Here $\mathcal{B}(\mathcal{H})$ denotes the algebra of all bounded linear operators acting on a complex Hilbert space $\mathcal{H}$. For $T \in \mathcal{B}(\mathcal{H})$, $\norm{T}$ denotes the operator norm and $T^{*}$ the usual adjoint, and we write $Re(T):=(T+T^*)/2$. If $z=(z_1,\ldots,z_n)$, then $\overline{z}$ denotes the componentwise complex conjugate $(\overline{z}_1,\ldots,\overline{z}_n)$.
	 
	 We denote by $\mathcal{PH}(\Omega,X)$ the space of all bounded $X$-valued pluriharmonic functions on $\Omega$, where $X=\mathcal{B}(\mathcal{H})$. The boundedness condition means that $\norm{f}_{\Omega,X}:=\sup_{z \in \Omega}\,\norm{f(z)}_{X} < \infty$ for every $f \in \mathcal{PH}(\Omega,X)$.
	\par
	Let $Y$ be any complex Banach space. Let $\Omega\subset \mathbb{C}^n$ be a simply connected complete Reinhardt domain and $n\in \mathbb{N}$. Let $U:\mathcal{B}(\mathcal{H})\rightarrow Y$ be a bounded liner operator and $\norm{U} \leq \lambda$. For $1 \leq p < \infty$, the first $\lambda$-powered Bohr radius of $U$ with respect to the class $\mathcal{PH}(\Omega,\mathcal{B}(\mathcal{H}))$, denoted by $R_{\lambda}(\Omega, p,U)$, is defined to be the supremum of all $r \in [0,1]$ such that for all $f \in \mathcal{PH}(\Omega,\mathcal{B}(\mathcal{H}))$ of the form \eqref{e-1.3-a} we have 
	\begin{equation} \label{e-1.4-a}
		\sup_{z \in r\Omega}\,\sum_{m=0}^{\infty} \sum_{|\alpha|=m} (\norm{U(a_{\alpha})}^p_{Y} + \norm{U(b_{\alpha})}^p_{Y})|z^\alpha|^p \leq \lambda^p\,\norm{f}^p_{\Omega,\mathcal{B}(\mathcal{H})}. 
	\end{equation}
	
	\noindent For short, write $R(\Omega, p,U):=R_{1}(\Omega, p,U)$, $R_{\lambda}(\Omega, p,\mathcal{B}(\mathcal{H})):=R_{\lambda}(\Omega, p,U)$ whenever $U=I:\mathcal{B}(\mathcal{H})\rightarrow \mathcal{B}(\mathcal{H})$, $R(\Omega, p,\mathcal{B}(\mathcal{H})):=R_{1}(\Omega, p,\mathcal{B}(\mathcal{H}))$, $R_{\lambda}(\Omega, p):=R_{\lambda}(\Omega, p,\mathbb{C})$, and $R(\Omega,p):=R_1(\Omega,p)$. The above definition was introduced by the present author of this article, who also undertook a detailed study of the asymptotic behaviour of $R_{\lambda}(\Omega, p,U)$ (see \cite{Himadri-local-Banach-1}). The Bohr phenomenon for harmonic functions of one variable in the scalar setting was first investigated in \cite{Abu-2010}. Subsequently, the asymptotic behaviour of the constants $R(\Omega,1)$ and $R(\Omega,p)$ was investigated in \cite{hamada-JFA-2022,das-2024}  whenever $\Omega=B_{\ell^n _q}$, where $B_{\ell^n _q}:=\left\{z=(z_{1},\ldots,z_{n}) \in \mathbb{C}^n:\norm{z}_{q}:=\left(\sum_{i=1}^{n}|z_{i}|^q\right)^{1/q}<1\right\}$, $1\leq q <\infty$ and $B_{\ell^n _\infty}:=\mathbb{D}^n$.
	\vspace{1mm}
	
	In a related but distinct approach, Aizenberg \cite{aizn-2000a} introduced another variant of multidimensional Bohr radius, namely the second Bohr radius, for complex valued holomorphic functions on bounded complete Reinhardt domains in $\mathbb{C}^n$. Primarily motivated by this notion, together with the above definition, we now introduce the concept of the second Bohr radius for operator valued pluriharmonic functions. Let $\Omega,\, Y,\, U$ and $\lambda$ be as in the above definition defined by \eqref{e-1.4-a}. For $1 \leq p < \infty$, the second $\lambda$-powered Bohr radius of $U$ with respect to the class $\mathcal{PH}(\Omega,\mathcal{B}(\mathcal{H}))$, denoted by $L_{\lambda}(\Omega, p,U)$, is defined to be the supremum of all $r \in [0,1]$ such that for all $f \in \mathcal{PH}(\Omega,\mathcal{B}(\mathcal{H}))$ of the form \eqref{e-1.3-a} we have 
	\begin{equation} \label{e-1.4-aaa}
		\sum_{m=0}^{\infty} \sum_{|\alpha|=m} \sup_{z \in r\Omega}\, (\norm{U(a_{\alpha})}^p_{Y} + \norm{U(b_{\alpha})}^p_{Y})|z^\alpha|^p \leq \lambda^p\,\norm{f}^p_{\Omega,\mathcal{B}(\mathcal{H})}. 
	\end{equation}
	
	\noindent For brevity, we adopt the following notations: $L(\Omega, p,U):=L_{1}(\Omega, p,U)$, $L_{\lambda}(\Omega, p,\mathcal{B}(\mathcal{H})):=L_{\lambda}(\Omega, p,U)$ whenever $U=I:\mathcal{B}(\mathcal{H})\rightarrow \mathcal{B}(\mathcal{H})$, $L(\Omega, p,\mathcal{B}(\mathcal{H})):=L_{1}(\Omega, p,\mathcal{B}(\mathcal{H}))$, $L_{\lambda}(\Omega, p):=L_{\lambda}(\Omega, p,\mathbb{C})$, and $L(\Omega,p):=L_1(\Omega,p)$.
	\vspace{1mm}
	
	In the same spirit as that of pluriharmonic functions, one can consider the corresponding notion of the second Bohr radius for Banach space valued holomorphic functions. Let $\Omega\subset \mathbb{C}^n$ be a complete Reinhardt domain and $n\in \mathbb{N}$. Let $U:X\rightarrow Y$ be a bounded liner operator between two complex Banach spaces and $\norm{U} \leq \lambda$. For $1 \leq p < \infty$, the second $\lambda$-powered Bohr radius of $U$ with respect to the class $\mathcal{H}(\Omega,X)$, denoted by $B_{\lambda}(\Omega, p,U)$, is defined as the supremum of all $r \in [0,1]$ such that for every $f \in \mathcal{H}(\Omega,X)$ of the form $f(z)=\sum_{\alpha}a_{\alpha}z^{\alpha}$ one has 
	\begin{equation} \label{e-1.2-a}
		 \sum_{\alpha} \sup_{z \in r\Omega}\, \norm{U(a_{\alpha})z^\alpha}^p_{Y} \leq \lambda^p\,\norm{f}^p_{\Omega,X}, 
	\end{equation}
	where $\norm{f}_{\Omega,X}:=\sup_{z \in \Omega}\norm{f(z)}_{X}$. Set $B(\Omega, p,U):=B_{1}(\Omega, p,U)$, $B_{\lambda}(\Omega, p,X):=B_{\lambda}(\Omega, p,U)$ whenever $U=I:X\rightarrow X$, $B(\Omega, p,X):=B_{1}(\Omega, p,X)$, $B_{\lambda}(\Omega, p):=B_{\lambda}(\Omega, p,\mathbb{C})$, and $B(\Omega,p):=B_1(\Omega,p)$. The constant $B(\Omega,1)$ was first investigated by Aizenberg \cite{aizn-2000a}, who derived both upper and lower bounds. Subsequent advances were made by Boas in \cite{boas-2000}. A detailed discussion of these results is presented in Section $3$.
	\vspace{1mm}
	
	An application of the elementary inequality stating that the supremum of a sum does not exceed the sum of the suprema yields $B_{\lambda}(\Omega, p,U) \leq K_{\lambda}(\Omega, p,U)$ and $L_{\lambda}(\Omega, p,U) \leq R_{\lambda}(\Omega, p,U)$. Moreover, equality holds in both cases whenever $\Omega=\mathbb{D}^n$, the unit polydisc. In addition, it is well known that $B(\mathbb{D},1)=K(\mathbb{D},1)=1/3$, which is precisely the classical Bohr radius of the unit disc $\mathbb{D}$.
	\vspace{1mm}
	
	It is appropriate to note that the Bohr radius problem becomes unstable when $\dim(X)>1$. This is evident from the fact that $K(\mathbb{D},1,B_{\ell^n_q})=0$ (see \cite{Blasco-OTAA-2010}). Furthermore, Blasco \cite{Blasco-Collect-2017} showed that $K(\mathbb{D},p,X)=0$ for $p \in [1,2)$ whenever $\dim (X) \geq 2$, whereas $K(\mathbb{D},p,X)>0$ if and only if $X$ is $p$-uniformly $\mathbb{C}$-convex for $2 \geq p < \infty$. These results also apply to $B(\mathbb{D},p,X)$, since $B_{\lambda}(\mathbb{D}^n, p,U) = K_{\lambda}(\mathbb{D}^n, p,U)$.
	The positivity of the
	constant $K_{\lambda}(\mathbb{D}^n, p,U)$ was investigated by Defant, Maestre, and Schwarting \cite{defant-2011} under suitable conditions. 
	\par
	Consequently, a central question is whether the first and second Bohr radii are non zero for a given complete Reinhardt domain $\Omega$. In a recent paper, the present author \cite{Himadri-local-Banach-1} has answered this question for the first Bohr radii constants $K_{\lambda}(\Omega, p,U)$ and $R_{\lambda}(\Omega, p,U)$ corresponding to holomorphic and pluharmonic functions, respectively.
	\vspace{0.5mm}
	
	 With the first Bohr radius problem now settled, it is natural to turn to the second Bohr radii and investigate under what conditions the constants $B_{\lambda}(\Omega, p,U)$ and $L_{\lambda}(\Omega, p,U)$ are non zero. This leads to the following general question.
	 \begin{qsn} \label{qsn-1.3}
	 	For a given complete Reinhardt domain $\Omega \subset \mathbb{C}^n$, can the second Bohr radii constants $B_{\lambda}(\Omega, p,U)$ and $L_{\lambda}(\Omega, p,U)$ be studied in detail? If so, under what conditions are these constants nonzero, and what can be said about their asymptotic behavior?
	 \end{qsn}  
	The main contribution of this paper is to answer this question affirmatively using the framework of local Banach space theory. To the best of our knowledge, these questions have not been previously explored. 
	\vspace{1mm}
	
	The paper is organized as follows. In Section $2$, we establish conditions under which the second Bohr radii constants $B_{\lambda}(\Omega, p,U)$ and $L_{\lambda}(\Omega, p,U)$ are strictly positive. Section $3$ is devoted to the asymptotic study of these constants using invariants from local Banach space theory in the case where $X$ or $\mathcal{B}(\mathcal{H})$ is finite-dimensional, while the infinite-dimensional setting is treated in Section $4$. Section $5$ presents several applications of our results to Banach sequence spaces.
	\comment{Over the last three decades, Bohr's theorem and its applications have occupied a central position in geometric function theory, connecting ideas from several complex variables, functional analysis, and operator theory. For a given complete Reinhardt domain $\Omega \subset \mathbb{C}^n$, the Bohr radius $K_{n}(\Omega)$ of $\Omega$ is defined to be the supremum of all $r \in [0,1]$ such that for each holomorphic function $f(z)=\sum_{\alpha}c_{\alpha}\,z^{\alpha}:\Omega\rightarrow \mathbb{C}$ we have
		\begin{equation} \label{e-hh-mult-BI}
			\sup_{z \in r\Omega}\, \sum_{\alpha}|c_{\alpha}\,z^{\alpha}| \leq \sup_{z \in \Omega}\, \left|f(z)\right|.
		\end{equation}
		Note that with this notation, Bohr's celebrated power series theorem states that $K_{1}(\mathbb{D})=1/3$ (see \cite{Bohr-1914}). One of the most striking aspects of Bohr's theorem arises in multidimensional settings, where the situation becomes significantly more complex and many problems remain open.}
	\section{Characterization of non zero second Bohr radii constants}
	A family of holomorphic functions defined on a complete Reinhardt domain $\Omega$ is said to satisfy the first Bohr phenomenon if there exists a constant $r=r_{0}\in (0,1]$ such that inequality \eqref{e-1.2} is valid for every function in the class. The largest such constant is referred to as the first Bohr radius associated with that class. An analogous notion can be formulated for pluriharmonic mappings on bounded simply connected domains: in this case, the first Bohr phenomenon is said to hold whenever \eqref{e-1.4-a} is satisfied for a universal radius $r=r_{0}\in (0,1]$. Replacing \eqref{e-1.2} and \eqref{e-1.4-a} with \eqref{e-1.2-a} and \eqref{e-1.4-aaa}, respectively, leads to the formulation of a second Bohr phenomenon in the holomorphic and pluriharmonic frameworks.
	\par 
	It is well known that the Bohr phenomenon does not hold universally across all function classes; numerous counterexamples and partial results illustrate this limitation (see \cite{aizn-2000b,Himadri-Vasu-PEMS,bene-2004,Blasco-Collect-2017,defant-2011,Himadri-local-Banach-1}). This naturally leads to the problem of identifying those classes of functions for which a Bohr phenomenon is valid. A complete characterization of this problem in the context of the first Bohr phenomenon was obtained by the author in \cite{Himadri-local-Banach-1}. The present work extends this line of investigation to the second Bohr phenomenon, providing conditions that guarantee its validity and, in doing so, resolving first part of Question \ref{qsn-1.3}.
	\vspace{1mm}

	We begin with the following lemma comparing second Bohr radius constants for pairs of domains, starting with the pluriharmonic case. 
	\par
	Before stating our result, we introduce the following standard notation. For bounded Reinhardt domains $\Omega_{1}, \Omega_{2}\subset \mathbb{C}^n$, let $S(\Omega_1, \Omega_{2}):= \inf \left\{s>0 : \Omega_{1} \subset s \, \Omega_{2}\right\}$. A Banach space $Z$ satisfying $\ell_1 \subset Z \subset c_{0}$ (with normal inclusions) is called a Banach sequence space if the canonical unit vectors ${e_{k}}$ form a $1$-unconditional basis of $Z$ (see Section $3$ for the definition of unconditional basis). Recall that if $Z$ and $W$ are Banach sequence spaces then $S(B_{Z},B_{W})=\norm{\mathrm{Id}:Z \rightarrow W}$.
	
	 \begin{lem} \label{lem-3.5}
		Let $\Omega_{1}$ and $\Omega_{2}$ be two bounded simply connected complete Reinhardt domains in $\mathbb{C}^n$. Then we have $$\frac{L_{\lambda}(\Omega_{2}, p,U)}{S(\Omega_{1},\Omega_{2})S(\Omega_{2},\Omega_{1})} \, \leq L_{\lambda}(\Omega_{1}, p,U) \leq S(\Omega_{1},\Omega_{2})S(\Omega_{2},\Omega_{1}) \, L_{\lambda}(\Omega_{2}, p,U).
		$$
	\end{lem}
\begin{proof}
We first prove the right hand inequality.
Let $f \in \mathcal{PH}(\Omega_2,\mathcal{B}(\mathcal{H}))$ be of the form \eqref{e-1.3-a} with $\|f\|_{\Omega_2} \le 1$.

Let $\delta_1>0$.  
By the definition of $S(\Omega_1,\Omega_2)$,
\[
z \in \Omega_1
\quad \Rightarrow \quad
\frac{z}{S(\Omega_1,\Omega_2)+\delta_1} \in \Omega_2.
\]

Define
\[
g(z)=f\!\left(
\frac{z}{S(\Omega_1,\Omega_2)+\delta_1}
\right).
\]
Then $g \in \mathcal{PH}(\Omega_1,\mathcal{B}(\mathcal{H}))$
and $\|g\|_{\Omega_1}\le 1$.
\par
Now, for fix $0<\varepsilon<L_{\lambda}(\Omega_1,p,U)$, let $r = L_{\lambda}(\Omega_1,p,U)-\varepsilon$.
Since $r<L_{\lambda}(\Omega_1,p,U)$, we have
\[
\sum_{m=0}^{\infty}\sum_{|\alpha|=m}
\sup_{z\in r\Omega_1}
(\|U(a_\alpha)\|^p+\|U(b_\alpha)\|^p)
\left|
\left(
\frac{z}{S(\Omega_1,\Omega_2)+\delta_1}
\right)^\alpha
\right|^p
\le \lambda^p.
\]

Hence
\begin{equation} \label{e-compare-1}
\sum_{m=0}^{\infty}\sum_{|\alpha|=m}
\sup_{z\in \frac{r}{S(\Omega_1,\Omega_2)+\delta_1}\Omega_1}
(\|U(a_\alpha)\|^p+\|U(b_\alpha)\|^p)
|z^\alpha|^p
\le \lambda^p.
\end{equation}

Now let $\delta_2>0$.  
Using the definition of $S(\Omega_2,\Omega_1)$, we have
\[
z \in \Omega_2
\quad \Rightarrow \quad
\frac{z}{S(\Omega_2,\Omega_1)+\delta_2} \in \Omega_1.
\]
Thus, from \eqref{e-compare-1}, we obtain
\[
\sum_{m=0}^{\infty}\sum_{|\alpha|=m}
\sup_{z\in
	\frac{r}{(S(\Omega_1,\Omega_2)+\delta_1)(S(\Omega_2,\Omega_1)+\delta_2)}
	\Omega_2}
(\|U(a_\alpha)\|^p+\|U(b_\alpha)\|^p)
|z^\alpha|^p
\le \lambda^p.
\]

By definition of $L_{\lambda}(\Omega_2,p,U)$,
\[
\frac{r}
{(S(\Omega_1,\Omega_2)+\delta_1)(S(\Omega_2,\Omega_1)+\delta_2)}
\le
L_{\lambda}(\Omega_2,p,U).
\]
Letting $\delta_1,\delta_2 \to 0$ gives
\[
r
\le
S(\Omega_1,\Omega_2)S(\Omega_2,\Omega_1)
L_{\lambda}(\Omega_2,p,U).
\]

Since $\varepsilon$ is arbitrary, we deduce that
\[
L_{\lambda}(\Omega_1,p,U)
\le
S(\Omega_1,\Omega_2)S(\Omega_2,\Omega_1)
L_{\lambda}(\Omega_2,p,U).
\]
Interchanging $\Omega_1$ and $\Omega_2$ in the preceding inequality gives the left hand inequality. 
\end{proof}
The above lemma can be viewed as an analogue of \cite[Lemma 3.1(1)]{defant-2004}, where the authors considered only scalar valued holomorphic functions, with $U$ being the identity operator on $\mathbb{C}$ and the power $p=1$. The proof follows the same pattern as that of \cite[Lemma 2.5]{defant-2004}. For the sake of completeness, we include the proof here. The main difference lies in the fact that we work with vector valued pluriharmonic functions associated with a bounded operator and an arbitrary power $p$.
\par	
We now turn our attention to the nonvanishing part of Question \ref{qsn-1.3}. By deriving an explicit positive lower bound, we prove that  $L_{\lambda}(\Omega, p,U)>0$ for every nonzero bounded linear operator $U: \mathcal{B}(\mathcal{H}) \rightarrow Y$ satisfying $\norm{U}< \lambda$ with $\lambda>1$.
	
	\begin{thm} \label{thm-1.1}
		Let $X=\mathcal{B}(\mathcal{H})$ and $U: X \rightarrow Y$ be a non-null bounded linear operator such that $\norm{U}< \lambda$. Then, for $\lambda>1$ and $n \in \mathbb{N}$, we have
		$$
		L_{\lambda}(\Omega, p,U) \geq D. \frac{1}{n^{\frac{1}{p}}}\,\frac{1}{S(\Omega,\mathbb{D}^n)S(\mathbb{D}^n,\Omega)},
		$$ where 
		$$
		D=\begin{cases}
			\max \left\{\frac{1}{4 \lambda\,2^{\frac{1}{p}}}\left(\frac{\lambda^p - \norm{U}^p}{2\lambda^p - \norm{U}}\right)^{\frac{1}{p}}\, , \frac{1}{4 \lambda\,2^{\frac{1}{p}}}\left(\frac{\lambda^p - \norm{U}^p}{\lambda^p - \norm{U}^p +1}\right)^{\frac{1}{p}}\, \frac{1}{\norm{U}}\right\}\,  & \text{for $\norm{U}\geq \frac{1}{4 \lambda\,2^{\frac{1}{p}}}$},\\[3mm]
			\max \left\{\frac{1}{4 \lambda\,2^{\frac{1}{p}}}\left(\frac{\lambda^p - \norm{U}^p}{2\lambda^p - \norm{U}}\right)^{\frac{1}{p}}, \frac{1}{4 \lambda\,2^{\frac{1}{p}}}\left(\frac{\lambda^p - \norm{U}^p}{\lambda^p - \norm{U}^p +1}\right)^{\frac{1}{p}}\right\} & \text{for $0<\norm{U}< \frac{1}{4 \lambda\,2^{\frac{1}{p}}}$}.
		\end{cases}
		$$
	\end{thm}
\begin{proof}[{\bf Proof}]
We begin by observing that $L_{\lambda}(\mathbb{D}^n, p,U)=R_{\lambda}(\mathbb{D}^n, p,U)$. It follows from \cite[Theorem 1.1]{Himadri-local-Banach-1} that for every $p \geq 1$,
\begin{equation*}
L_{\lambda}(\mathbb{D}^n, p,U)=R_{\lambda}(\mathbb{D}^n, p,U) \geq D.\, \frac{1}{\sup_{z \in \mathbb{D}^n}\norm{z}_{p}},
\end{equation*}
where $D$ as in the statement of Theorem \ref{thm-1.1}. An application of H\"{o}lder's inequality yields $\sup_{z \in \mathbb{D}^n}\norm{z}_{p}=n^{1/p}$ for any $1 \leq p < \infty$. Substituting this into the above inequality, we obtain
\begin{equation*}
	L_{\lambda}(\mathbb{D}^n, p,U) \geq D.\, n^{-\frac{1}{p}}.
	\end{equation*}
We now invoke Lemma \ref{lem-3.5}. Choosing $\Omega_{1}=\Omega$ and $\Omega_{2}=\mathbb{D}^n$ in that lemma, and combining it with the preceding estimate, we conclude that
\begin{equation*}
L_{\lambda}(\Omega, p,U) \geq  \frac{L_{\lambda}(\mathbb{D}^n, p,U)}{S(\Omega,\mathbb{D}^n)S(\mathbb{D}^n,\Omega)} \geq  D.\, n^{-\frac{1}{p}}\,\frac{1}{S(\Omega,\mathbb{D}^n)S(\mathbb{D}^n,\Omega)}. 
\end{equation*}
The proof is therefore complete.
\end{proof}	
Applying this theorem to the case $\Omega=B_{\ell^n_q}$, the Minkowski space, yields the following result.
\begin{cor}
	Let $Y,U$ and $\lambda$ be as in Theorem \ref{thm-1.1}. Then, for $1 \leq p < \infty$, $1 \leq q \leq \infty$ and $\lambda>1$,
	\begin{equation*}
		L_{\lambda}(B_{\ell^n_q}, p,U) \geq D\, \frac{1}{n^{\frac{1}{p}+\frac{1}{q}}},
	\end{equation*}
	where the constant $D$ is the same as in Theorem \ref{thm-1.1}.
\end{cor}	
	On the other hand, we obtain the following lower bound of the second $\lambda$-powered Bohr radius for operator-valued pluriharmonic functions whenever $U$ is the identity operator on $X=\mathcal{B}(\mathcal{H})$.
	\begin{cor}
		Let $X=\mathcal{B}(\mathcal{H})$ and $1<\lambda$. Then for all $n\in \mathbb{N}$ and $1\leq p < \infty$, we have
		\begin{equation*}
			L_{\lambda}(\Omega, p) \geq \frac{1}{2^{2+\frac{1}{p}}} \frac{\left(\lambda^p -1\right)^{\frac{1}{p}}}{\lambda^2}\, \frac{1}{n^{\frac{1}{p}}}\,\frac{1}{S(\Omega,\mathbb{D}^n)S(\mathbb{D}^n,\Omega)}.
		\end{equation*}
	\end{cor}
	\begin{rem}
		A natural question is whether Theorem \ref{thm-1.1} continues to hold when $\norm{U}=\lambda$. The example below demonstrates that, in this case, the constant $L_{\lambda}(\Omega, p,U)$ need not be strictly positive. Consequently, the assumption $\norm{U}<\lambda$ cannot be relaxed.
	\end{rem}
	\begin{example} \label{example-1.1}
		For a bounded simply connected complete Reinhardt domain $\Omega \subseteq \mathbb{C}^n$ and $k \in \mathbb{N}$, define 
		$F_{k}: \Omega \rightarrow \mathcal{B}(\mathcal{H})$ by $$F_{k}(z)=(i\, \cos \frac{1}{k}) I_{\mathcal{H}} + (\frac{1}{2} \sin \frac{1}{k})  I_{\mathcal{H}}z_{1}+(\frac{1}{2} \sin \frac{1}{k})  I_{\mathcal{H}} \overline{z_{1}}$$
		for $z=(z_{1}, \ldots,z_{n}) \in \Omega$, where $I_{\mathcal{H}}$ is the identity on $\mathcal{H}$. Let $U=\lambda \, I$, with $I$ the identity on $\mathcal{B}(\mathcal{H})$. Then $F_{k}(0)=(i\, \cos \frac{1}{k})  I_{\mathcal{H}}$. Clearly, $F_{k} \in \mathcal{PH}(\Omega,\mathcal{B}(\mathcal{H}))$ and $\norm{F_{k}}_{\Omega,\mathcal{B}(\mathcal{H})} \leq 1$. It is shown in \cite{Himadri-local-Banach-1} that there does not exist any $r_{0}>0$ such that \eqref{e-1.4-a} holds for all $f \in \mathcal{PH}(\Omega,\mathcal{B}(\mathcal{H}))$ for all $z \in r_{0}\, \Omega$. That is to say, $R_{\lambda}(\Omega, p,U)$ becomes zero. Since $L_{\lambda}(\Omega, p,U) \leq R_{\lambda}(\Omega, p,U)$, it immediately follows that $L_{\lambda}(\Omega, p,U)=0$. Therefore, we see that there is no non zero $r_{0}$ such that the inequality \eqref{e-1.4-aaa} holds for the functions $F_k$. This fact, in turn, implies that condition $\norm{U}<\lambda$ in Theorem \ref{thm-1.1} is necessary.
	\end{example}

We now turn to the second Bohr phenomenon for holomorphic functions. It is convenient to first clarify several conceptual differences between the study of Bohr’s theorem for vector valued holomorphic functions and that for pluriharmonic functions. In the pluriharmonic framework, attention is restricted to simply connected complete Reinhardt domains in $\mathbb{C}^n$. This assumption is imposed by the fact that the series representation \eqref{e-1.3-a} is valid only on such domains. In addition, we restrict our attention to pluriharmonic mappings taking values in $\mathcal{B}(\mathcal{H})$, and the associated operator $U$ is accordingly taken to act on $\mathcal{B}(\mathcal{H})$. This choice is motivated by the fact that complex conjugation is well defined in $\mathcal{B}(\mathcal{H})$, a property essential to our analysis. 
\vspace{1mm}

In contrast, the vector valued holomorphic framework is more flexible. Complete Reinhardt domains suffice, without assuming simple connectedness, because multivariable power series converge exactly on such domains and yield holomorphic functions on their domain of convergence. Moreover, since complex conjugation plays no role in this setting, there is no need to restrict attention to $\mathcal{B}(\mathcal{H})$-valued functions. Consequently, holomorphic functions with values in arbitrary complex Banach spaces may be considered.
\vspace{1mm}

We therefore study the Bohr radius problem for holomorphic functions within the following framework: bounded holomorphic mappings with values in an arbitrary complex Banach space $X$, defined on complete Reinhardt domains in $\mathbb{C}^n$. 
\vspace{1mm}

Following similar arguments to that given in the proof of Lemma \ref{lem-3.5}, one can show the following result, which compares the second Bohr radius constants of holomorphic functions for two domains. 
\begin{lem} \label{lem-3.5-a}
	Let $\Omega_{1}$ and $\Omega_{2}$ be two bounded complete Reinhardt domains in $\mathbb{C}^n$. Then we have $$\frac{B_{\lambda}(\Omega_{2}, p,U)}{S(\Omega_{1},\Omega_{2})S(\Omega_{2},\Omega_{1})} \, \leq B_{\lambda}(\Omega_{1}, p,U) \leq S(\Omega_{1},\Omega_{2})S(\Omega_{2},\Omega_{1}) \, B_{\lambda}(\Omega_{2}, p,U).
	$$
\end{lem}
This lemma can be viewed as a vector valued analogue of \cite[Lemma 3.1]{defant-2004}, where the authors focused on only scalar valued holomorphic functions.
\par
In the following theorem, we establish that the constant $B_{\lambda}(\Omega, p, U)$ is nonzero, which serves as the holomorphic analogue of Theorem \ref{thm-1.1}. However, the proof is similar to that of Theorem \ref{thm-1.1} and hence, omitted.
\begin{thm} \label{thm-1.1-a}
	Let $X$ and $Y$ be any complex Banach spaces and $U: X \rightarrow Y$ be a non-null bounded linear operator such that $\norm{U}< \lambda$. Then, for $\lambda>1$ and $n \in \mathbb{N}$, we have
	$$
	B_{\lambda}(\Omega, p,U) \geq  C.\,\frac{1}{n^{\frac{1}{p}}}\,\frac{1}{S(\Omega,\mathbb{D}^n)S(\mathbb{D}^n,\Omega)},$$
	where 
	$	C=\begin{cases}
		\max \left\{\left(\frac{\lambda^p - \norm{U}^p}{2\lambda^p - \norm{U}}\right)^{\frac{1}{p}}\, , \left(\frac{\lambda^p - \norm{U}^p}{\lambda^p - \norm{U}^p +1}\right)^{\frac{1}{p}}\, \frac{1}{\norm{U}}\right\}\,  & \text{for $\norm{U}\geq 1$},\\[3mm]
		\max \left\{\left(\frac{\lambda^p - \norm{U}^p}{2\lambda^p - \norm{U}}\right)^{\frac{1}{p}}, \left(\frac{\lambda^p - \norm{U}^p}{\lambda^p - \norm{U}^p +1}\right)^{\frac{1}{p}}\right\} & \text{for $0<\norm{U}< 1$}.
	\end{cases}
	$
\end{thm}
An immediate application of this theorem gives the following result for $\Omega=B_{\ell^n_q}$, the Minkowski space.
\begin{cor}
Let $X,Y,U$ and $\lambda$ be as in Theorem \ref{thm-1.1-a}. Then, for $1 \leq q \leq \infty$ and $\lambda>1$,
\begin{equation*}
B_{\lambda}(B_{\ell^n_q}, p,U) \geq C\, \frac{1}{n^{\frac{1}{p}+\frac{1}{q}}},
\end{equation*}
where the constant $C$ is the same as in Theorem \ref{thm-1.1-a}.
\end{cor}
It is worth noting that \cite[Proposition 3.3]{defant-2012} follows as a direct consequence of this theorem by choosing $p=1$. Furthermore, as an immediate corollary to the above Theorem, we obtain the following lower bound for the second $\lambda$-powered Bohr radius of $X$ valued holomorphic functions in the case where $U$ is the identity operator on $X$.
\begin{cor}
	Let $1<\lambda$. Then for all $n\in \mathbb{N}$ and $1\leq p < \infty$, we have
	\begin{equation*}
		B_{\lambda}(\Omega, p,X) \geq  \frac{\left(\lambda^p -1\right)^{\frac{1}{p}}}{\lambda}\, \frac{1}{n^{\frac{1}{p}}}\,\frac{1}{S(\Omega,\mathbb{D}^n)S(\mathbb{D}^n,\Omega)}.
	\end{equation*}
\end{cor}
In particular, when $\Omega=B_{\ell^n_q}$, the Minkowski space, an application of this corollary yields the following result.
\begin{cor}
	Let $1<\lambda$ and $1\leq p < \infty$. Then for all $n\in \mathbb{N}$ and $1\leq q \leq \infty$, we have
	\begin{equation*}
		B_{\lambda}(B_{\ell^n_q}, p,X) \geq  
		\frac{\left(\lambda^p -1\right)^{\frac{1}{p}}}{\lambda}\, \frac{1}{n^{\frac{1}{p}+\frac{1}{q}}}.
	\end{equation*}
\end{cor}
By setting $p=1$ in this corollary, we recover \cite[Corollary 3.4]{defant-2012}, which we record below for the sake of completeness.
\begin{cor}
	Let $1<\lambda$ and $1\leq p < \infty$. Then for all $n\in \mathbb{N}$ and $1\leq q \leq \infty$, we have
	\begin{equation*}
	K_{\lambda}(\mathbb{D}^n, 1,X)=	B_{\lambda}(\mathbb{D}^n, 1,X) \geq  
		\frac{\left(\lambda -1\right)}{\lambda}\, \frac{1}{n}.
	\end{equation*}
\end{cor}
\section{Asymptotic estimates of the second Bohr radii constants in finite dimensional settings}
	We now turn to the second part of Question \ref{qsn-1.3}, which is the central focus of this paper, namely obtaining asymptotic estimates for the second Bohr radii constants $L_{\lambda}(\Omega, p,U)$ and $B_{\lambda}(\Omega, p,U)$. It is already known that second Bohr radii constants are dominated by their first Bohr radii counterparts, in the sense that $L_{\lambda}(\Omega, p,U) \leq R_{\lambda}(\Omega, p,U)$ and $B_{\lambda}(\Omega, p,U) \leq K_{\lambda}(\Omega, p,U)$. Consequently, any upper bounds established for $R_{\lambda}(\Omega, p,U)$ and $K_{\lambda}(\Omega, p,U)$ immediately yield corresponding bounds for $L_{\lambda}(\Omega, p,U)$ and $B_{\lambda}(\Omega, p,U)$. For this reason, it is natural to begin with a brief overview of recent developments on first Bohr radius constants.
	\par
	The asymptotic behavior of the first $\lambda$-Bohr radius $K_{\lambda}(\mathbb{D}^n, 1,U)$ for holomorphic functions is well understood and was determined in \cite{defant-2012}. Nevertheless, despite considerable progress, the precise value of the constant $K_{n}(\Omega)$ remains poorly understood. In particular, its exact value is unknown for dimensions $n>1$, even in the polydisc $\mathbb{D}^n$, where $\mathbb{D}^n:=\{z=(z_{1},\ldots,z_{n}) \in \mathbb{C}^n: |z_{j}|<1\, \mbox{for all}\, 1 \leq j \leq n\}$. It is known from \cite{boas-1997} that $K_n(\mathbb{D}^n)\leq 1/3$, and that the sequence $\{K_n(\mathbb{D}^n)\}_{n=1}^{\infty}$ is decreasing and converges to zero as $n \rightarrow \infty$. The authors of this paper also have established the following estimate for $n\geq 2$,
	$$
	\frac{1}{3\sqrt{n}} \leq K_n(\mathbb{D}^n) \leq 2\sqrt{\frac{\log\,n}{n}}.
	$$
	In particular, when	$n=2$
	this yields the bounds $0.23570226\leq K_2(\mathbb{D}^2) \leq 1/3=0.33333$. After nearly three decades, these estimates have been improved in recent years. Specifically, Knese \cite{knese-2025} showed in $2025$ that $K_2(\mathbb{D}^2)>0.3006$, while Baran, Pikul, Woerdeman, and Wojtylak \cite{baran-2026} proved in $2026$ that $K_2(\mathbb{D}^2) < 0.3177$. Taken together, these results confirm in particular that $K_n(\mathbb{D}^n) <1/3$ for $n \geq 2$.
	\vspace{1mm}
	
	Since the exact value of $K_{n}(\Omega)$ is unknown, attention naturally turns to its asymptotic behavior on general domains, especially on key domains such as the unit ball of Minkowski space and finite-dimensional complex Banach spaces. Beginning with the work of Dineen and Timoney \cite{Dineen-Timoney-1989}, and subsequently developed by Boas and Khavinson \cite{boas-1997,boas-2000}, with further extensions by Aizenberg \cite{aizn-2000a} and by Defant and Frerick \cite{defant-2011-lpn}, it has been shown that for every $1\leq q \leq \infty$ and all $n\in \mathbb{N}$, there exist constants $C,D>0$ such that 
	\begin{equation*}
		\frac{1}{C}\, \left(\frac{\log\,n}{n}\right)^{1-\frac{1}{\min\{q,2\}}} \leq K_{n}(B_{\ell^n _q}) \leq D \,\left(\frac{\log\,n}{n}\right)^{1-\frac{1}{\min\{q,2\}}}.
	\end{equation*}
	The one-dimensional constant $K_{\lambda}(\mathbb{D}, 1)$ was first studied by Bombieri \cite{bombieri-1962}, who determined its exact value for $\lambda \in [1,\sqrt{2}]$. Its precise asymptotic behavior as $\lambda\rightarrow \infty$ was later analyzed by Bombieri and Bourgain \cite{bombieri-2004}. Building on these developments, Defant {\it et al.} have introduced the more general constant $K_\lambda(\mathbb{D}^n,1,U)$ in \cite{defant-2012} and established its asymptotic estimates in both finite- and infinite-dimensional Banach spaces $X$. 
	\par
	On a related note, the constants $K(\mathbb{D},p)$ and $K(\mathbb{D}^n,p)$ were first studied by Djakov and Ramanujan \cite{Djakov & Ramanujan & J. Anal & 2000}, and later analogues of these results in Hardy spaces were developed by B\'{e}n\'{e}teau, Dahlner, and Khavinson \cite{bene-2004}. More recently, it has been shown in \cite{Himadri-local-Banach-1} that the constants $R_{\lambda}(\Omega, p,U)$ and $K_{\lambda}(\Omega, p,U)$ have the same asymptotic order up to multiplicative constants.
	\vspace{1mm}
	
	On the other hand, the precise values of the second Bohr radius constants in higher dimensions are still unknown. Clearly, $K_{\lambda}(\mathbb{D}, p,U)=B_{\lambda}(\mathbb{D}, p,U)$, and hence, by \cite{bombieri-2004}, one has
	$$
	B_{\lambda}(\mathbb{D}, 1) =\frac{1}{3\lambda- 2\,\sqrt{2\,(\lambda^2 -1)}}\,\,\, \mbox{for all}\,\,\, 1\leq \lambda \leq \sqrt{2}.
	$$
	\noindent Although several notable advances have been made, the multidimensional second Bohr radius problem remains substantially more complicated than its one-dimensional counterpart. In this direction, the first progress was achieved by Aizenberg \cite{aizn-2000a}, who showed that for any bounded complete Reinhardt domain $\Omega \subset \mathbb{C}^n$, with $n \geq 2$,
	\begin{align*}
	B(\Omega, 1) \geq 1- \left(\frac{2}{3}\right)^{\frac{1}{n}}
	\end{align*}
	and moreover that $B(B_{\ell^n_1}, 1)< 0.446663/n$. Furthermore, due to Boas \cite{boas-2000}, it is known that for $1\leq q \leq \infty$ and $n>1$, 
	\begin{equation} \label{e-lower-snd}
		\frac{1}{3}\, \left(\frac{1}{n}\right)^{\frac{1}{2}+\frac{1}{\max\{q,2\}}} \leq B(B_{\ell^n_q},1) \leq 4 \,\left(\frac{\log\,n}{n}\right)^{\frac{1}{2}+\frac{1}{\max\{q,2\}}}.
	\end{equation}
In contrast, the present article introduces the second Bohr radius constant $L_{\lambda}(\Omega, p,U)$ for pluriharmonic functions, a notion that, to the best of our knowledge, has not been previously investigated. Consequently, the exact value of this constant is not yet known. A natural and fundamental question, therefore, concerns the asymptotic behavior of these constants. Addressing this problem, however, turns out to be highly nontrivial and poses significant challenges.
	\vspace{1mm}	
		
	A main contribution of this paper is the determination of the asymptotic behavior of the second $\lambda$-Bohr radius for vector valued holomorphic and pluriharmonic functions on any complete Reinhardt domains using tools from local Banach space theory.
	\vspace{1mm}
	
	We now collect several facts that will be instrumental in what follows. Let $Z=(\mathbb{C}^n, ||.||)$ be an $n$-dimensional complex Banach space for which the canonical basis vectors $e_{k}$ form a normalized $1$-unconditional basis. Equivalently, the open unit ball $B_{Z}$ of $Z$ is a complete Reinhardt domain in $\mathbb{C}^n$. There is a one-to-one correspondence between bounded convex complete Reinhardt domains $\Omega \subseteq\mathbb{C}^n$ and the open unit balls of norms on $\mathbb{C}^n$ for which the canonical basis vectors $e_{k}$ form a normalized $1$-unconditional basis. Indeed, given such a domain $\Omega$, its Minkowski functional $p_{\Omega}:\mathbb{C}^n \rightarrow \mathbb{R}_{+}$ defined by $$p_{\Omega}(z):= \inf \left\{t>0 : z/t \in \Omega\right\}$$
	is a norm on $\mathbb{C}^n$, and $\Omega$ coincides with the open unit ball of the Banach space $(\mathbb{C}^n,p_{\Omega})$. Conversely, if $Z=(\mathbb{C}^n, ||.||)$ is a Banach space whose canonical basis is normalized $1$-unconditional, then its open unit ball $B_{	Z}$ is necessarily a bounded convex complete Reinhardt domain in $\mathbb{C}^n$. 
	\par
	As a consequence, investigating the constants $B_{\lambda}(\Omega, p,U)$ and $L_{\lambda}(\Omega, p,U)$ for unit balls $\Omega=B_{Z}$ of finite dimensional complex Banach spaces $(\mathbb{C}^n, ||.||)$ with the normalized $1$-unconditional canonical bases is equivalent to studying these constants over bounded complete Reinhardt domains $\Omega \subseteq\mathbb{C}^n$.
	\vspace{1mm}
	
	 We now proceed to state our first result in this section, which is the asymptotic behavior of the constant $L_{\lambda}(\Omega, p,U)$ when $\Omega=B_{Z}$ and $U$ is the identity operator on $X=\mathcal{B}(\mathcal{H})$. Recall that a Schauder basis $\{w_{k}\}$ of a Banach space $W$ is said to be $1$-unconditional if $\chi(\{w_k\})=1$ (see more details in subsection $3.1$). We usually denote the canonical basis vectors of $Z=(\mathbb{C}^n, ||.||)$ by $e_{k}$, $k=1, \ldots,n$. 
	\begin{thm} \label{thm-1.2}
		Let $Z=(\mathbb{C}^n, ||.||)$ be a Banach space such that $\chi(\{e_k\}^n_{k=1})=1$. Let $X=\mathcal{B}(\mathcal{H})$ be a finite dimensional complex Banach space and $\lambda>1$. Then for $1 \leq p < \infty$, $1 \leq q \leq \infty$,
		\begin{equation*}
			L_{\lambda}(B_{Z}, p,X) \geq	\begin{cases}
				E_{1}(X)\,\frac{1}{s\,t\, n^{\frac{1}{q}}}  \max \left\{\frac{1}{\sqrt{n}\, \norm{\mathrm{Id}: \ell^n_{2} \rightarrow Z}}, \, \frac{1}{e\, \norm{\mathrm{Id}: Z \rightarrow \ell^n_{1}}}\right\}\, \left(\frac{\lambda^p -1}{2\lambda^p - 1}\right)^{\frac{1}{p}} & \text{for $p=1$}, \\[2mm]
				E_{3}\, \frac{1}{s}\,\norm{\mathrm{Id}:\ell^n _2 \rightarrow Z}^{-\frac{2}{p}}\, \left(\frac{\lambda^p -1}{2\lambda^p - 1}\right)^{\frac{1}{p}}\, & \text{for $p \geq 2$}, \\[2mm]
				
				\frac{E_{2}(X)}{s\left(t n^{\frac{1}{q}}\right)^{\frac{(2-p)}{p}}} \max \left\{\frac{1}{(\sqrt{n})^{1-\theta} \norm{Id: \ell^n_{2} \rightarrow Z}},  \frac{\norm{\mathrm{Id}:\ell^n _2 \rightarrow Z}^{-\theta}}{e^{(1-\theta)} \norm{\mathrm{Id}: Z \rightarrow \ell^n_{1}}^{1-\theta}}\right\}\left(\frac{\lambda^p -1}{2\lambda^p - 1}\right)^{\frac{1}{p}} \hspace{-3mm} & \text{for $1<p< 2$}, 
			\end{cases}
		\end{equation*}
		where $\theta = (2(p-1))/p$, $s:=\norm{\mathrm{Id}: Z \rightarrow \ell^n_q}$, and $t:=\norm{\mathrm{Id}:  \ell^n_q\rightarrow Z}$. Here, $E_{1}(X)$, $E_{2}(X)$, and $E_{3}$ are positive constants: $E_{1}(X)$ depends only on $X$, $E_{2}(X)$ depends on both $X$ and $p$, while $E_{3}$ is independent of $X$ and depends only on $p$. Furthermore,
		\begin{equation*}
			L_{\lambda}(B_{Z}, p,X) \leq d\,\lambda^{\frac{1}{\log \,n}}\,b_{[\log\,n]}(B_Z)\, \,\frac{\norm{\mathrm{Id}:Z \rightarrow \ell^n_2}}{n^{\frac{1}{p}}}
		\end{equation*}
		for some constant $d>0$. Here $b_m(\Omega):=\left(\inf_{|\alpha|=m}\sup_{z \in \Omega} |z^{\alpha}|\right)^{-\frac{1}{m}}$, $m \in \mathbb{N}$ for a complete Reinhardt domain $\Omega$.
	\end{thm}
\begin{rem}
	The correspondence preceding Theorem \ref{thm-1.2} allows for an extension of Theorem \ref{thm-1.2} to bounded, simply connected, convex complete Reinhardt domains. Furthermore, by combining Theorem \ref{thm-1.2} with Lemma \ref{lem-3.5}, we obtain asymptotic estimates for all bounded, simply connected complete Reinhardt domains, not necessarily convex. It is worth noting that Theorem \ref{thm-1.2} may be regarded as a pluriharmonic analogue of \cite[Proposition 3.2 and Theorem 3.3]{defant-2004}, where the authors treated only scalar-valued holomorphic functions with the power $p=1$.
\end{rem}
In particular, for the Minkowski spaces $Z=\ell^n _q$, by proceeding along the same lines as in the proof of Theorem \ref{thm-1.2} and using the estimates obtained in Lemma \ref{lem-3.1}, we derive the following bounds. We omit the details.
\begin{cor} \label{cor-1.4-snd}
	Let $X=\mathcal{B}(\mathcal{H})$ be finite dimensional and $\lambda>1$. Then
	\begin{equation*}
		L_{\lambda}(B_{\ell^n _q}, p,X) \geq	\begin{cases}
			E'_{3}(X)\,\left(\frac{\lambda^p -1}{2\lambda^p - 1}\right)^{\frac{1}{p}} \, \frac{1}{n^{1/q}} \left(\frac{\log n}{n}\right)^{ \left(1- \frac{1}{\min\{q,2\}}\right)} & \text{for $p=1$}, \\[2mm]
			E'_{2} \left(\frac{\lambda^p -1}{2\lambda^p - 1}\right)^{\frac{1}{p}}\, n^{-\frac{1}{p}} & \text{for $p \geq q$}, \\[2mm]
			
			E_{4}(X)\,  \left(\frac{\lambda^p -1}{2\lambda^p - 1}\right)^{\frac{1}{p}} \,n^{-\left(\frac{p-1}{p(q-1)} + \frac{2-p}{pq}\right)}\, \left( \frac{\log n}{n}\right)^{\left(1- \frac{1}{\min\{q,2\}}\right)\, \frac{q-p}{p(q-1)}} & \text{for $1<p< q$}. 
		\end{cases}
	\end{equation*}
	Here, $E'_{3}(X)$, $E'_{2}$, and $E_{4}(X)$ are positive constants: $E'_{3}(X)$ depends only on $X$, $E_{4}(X)$ depends on both $X$ and $p$, while $E'_{2}$ is independent of $X$ and depends only on $p$. Furthermore,
	\begin{equation*}
		L_{\lambda}(B_{\ell^n_q}, p,X) \leq d\,\lambda^{\frac{1}{\log \,n}}\,b_{[\log\,n]}(B_{\ell^n_q})\, \,\frac{\norm{\mathrm{Id}:\ell^n_q \rightarrow \ell^n_2}}{n^{\frac{1}{p}}}
	\end{equation*}
	for some constant $d>0$.
\end{cor}
\begin{rem}
We now highlight two important remarks. First, by arguments analogous to those used in the proof of Theorem \ref{thm-1.2}, one can investigate the asymptotic behaviour of the second Bohr radius constant $B_{\lambda}(B_{Z}, p,X)$ for holomorphic functions and show that the asymptotic behaviors of $L_{\lambda}(B_{Z}, p,X)$ and $B_{\lambda}(B_{Z}, p,X)$ coincide up to a multiplicative constant. 
\par
Second, we note a difference between the lower bound of $L_{\lambda}(B_{\ell^n_q}, p,X)$ obtained in Corollary \ref{cor-1.4-snd} and the one given in \eqref{e-lower-snd}. In the lower bound appearing in \eqref{e-lower-snd}, no logarithmic term occurs, whereas in our result in Theorem \ref{thm-1.2} the lower bound contains a logarithmic factor. A similar phenomenon was also observed by the author in \cite{das-2024-CMB}, where only scalar valued holomorphic functions were considered, with $U$ being the identity on $\mathbb{C}$, and with $\lambda=1$ and $p=1$.
\end{rem}
\subsection{Basic definitions and notations}
Throughout the paper we employ standard notation and terminology from Banach space theory. All Banach spaces $W$ are complex. Their topological duals are denoted by $W^{*}$, and their open unit balls by $B_{W}$. A Banach space $W$ is said to have cotype $t\in [0,\infty]$ if there exists a constant $C>0$ such that for every finite family of vectors $w_{1},\ldots,w_{n}\in W$, 
$$\left(\sum_{k=1}^{n}\norm{w_{k}}^t\right)^{\frac{1}{t}} \leq C \left(\int_{0}^{1}\norm{\sum_{k=1}^{n}r_{k}(s)w_{k}}^2\, ds\right)^{\frac{1}{2}},$$
where $\{r_{k}\}$ denotes the sequence of Rademacher functions on $[0,1]$. We define 
$$\operatorname{Cot}(W):=\inf\{2 \leq t \leq \infty: W\, \mbox{has cotype}\, t\}.
$$
It is well known that every Banach space has cotype $\infty$. We adopt the convention $1/\operatorname{Cot}(W)=0$ whenever $\operatorname{Cot}(W)=\infty$. 
\par 
A Schauder basis $\{w_{k}\}$ of a Banach space $W$ is called unconditional if there exists a constant $c\geq 1$ such that $$\norm{\sum_{j=1}^{k}\epsilon_{j}\mu_{j}w_{j}} \leq c \norm{\sum_{j=1}^{k}\mu_{j}w_{j}}
$$
for all $k \in \mathbb{N}$, all scalars $\mu _{j} \in \mathbb{C}$, and all $\epsilon_{j} \in \mathbb{C}$ with $|\epsilon_{j}| \leq 1$.  The smallest such constant is denoted by $\chi(\{w_{k}\})$ and is called unconditional basis constant of $\{w_{k}\}$. The basis is said to be $1$-unconditional if $\chi(\{w_{k}\})=1$. Moreover, the unconditional basis constant of a Banach space $W$ is defined by $\chi(W):=\inf \chi(\{w_{k}\})\in [1,\infty]$, where the infimum is taken over all unconditional bases $\{w_{k}\}$ of $W$. 
\par
Let $Z=(\mathbb{C}^n, ||.||)$ be a Banach space and $Y$ any Banach space. Fix $m \in \mathbb{N}$. We denote by $\mathcal{P}(^m Z,Y)$ the space of all $m$-homogeneous polynomials $Q:Z \rightarrow Y$ of the form $Q(z)=\sum_{|\alpha|=m}c_{\alpha}z^{\alpha}$ endowed with the norm $\norm{Q}_{\mathcal{P}(^m Z,Y)}:=\sup_{z \in B_{Z}}|Q(z)|$. Furthermore, if $\Omega \subset \mathbb{C}^n$ is a bounded complete Reinhardt domain, then denote by $\mathcal{P}(^m \Omega,Y)$ the space of all $m$-homogeneous $Y$ valued pluriharmonic polynomials $Q$ considered as defined on $\Omega$.  The unconditional basis constant of the monomial basis $\{z^{\alpha}: |\alpha|=m\}$ in $\mathcal{P}(^m Z,Y)$ is denoted by $\chi_{M}(\mathcal{P}(^m Z,Y))$.
Denote by $\mathcal{PH}(^m Z,\mathcal{B}(H))$ the space of all $m$-homogeneous operator valued pluriharmonic polynomials $P:Z \rightarrow \mathcal{B}(H)$ of the form 
\begin{equation} \label{e-1.3-aa}
	P(z)= \sum_{|\alpha|=m} a_{\alpha}\, z^{\alpha} +  \sum_{|\alpha|=m} b^{*}_{\alpha}\, \bar{z}^{\alpha}
\end{equation} 
and set $\norm{P}_{\mathcal{PH}(^m Z,\mathcal{B}(\mathcal{H}))}:=\sup_{z \in B_{Z}}||P(z)||_{\mathcal{B}(\mathcal{H})}$. Moreover, if $\Omega \subset \mathbb{C}^n$ is a bounded simply connected complete Reinhardt domain, then denote by $\mathcal{PH}(^m \Omega,\mathcal{B}(\mathcal{H}))$ the space of all $m$-homogeneous operator valued pluriharmonic polynomials $P$ considered as defined on $\Omega$ of the form \eqref{e-1.3-aa}. 
\par A bounded linear operator $U:W\rightarrow Y$ between Banach spaces $W$ and $Y$ is called $(r,t)$-summing, with $r,t \in [1,\infty)$, if there exists a constant $C>0$ such that for every choice of finite family $w_{1}, \ldots,w_{n} \in W$, one has
$$\left(\sum_{k=1}^{n} \norm{U(w_{k})}^r\right)^{1/r} \leq C \, \sup_{\phi \in B_{X^{*}}} \left(\sum_{k=1}^{n} |\phi(w_{k})|^t\right)^{1/t}.
$$
The smallest constant $C$ satisfying this inequality is denoted by $\pi_{r,t}(U)$. In the special case $r=t$, the operator $U$ is called $r$- summing, and we write $\pi_{r}(U):=\pi_{r,r}(U)$.
\subsection{Second Bohr radii for homogeneous plynomials}
In the study of first Bohr radii, a standard approach for obtaining nontrivial estimates is to begin with the study of $m$-homogeneous polynomials. The following definition introduces the $m$-homogeneous analogue of second Bohr radii for holomorphic and pluriharmonic functions, which will serve as a foundation for our subsequent developments. Let $\Omega\subset \mathbb{C}^n$ be a simply connected complete Reinhardt domain and $n\in \mathbb{N}$. Let $U:\mathcal{B}(\mathcal{H})\rightarrow Y$ be a bounded liner operator and $\norm{U} \leq \lambda$. For $1 \leq p < \infty$, the second $\lambda$-powered Bohr radius of $U$ with respect to $\mathcal{PH}(^m \Omega,\mathcal{B}(\mathcal{H}))$, denoted by $L^m_{\lambda}(\Omega, p,U)$, is defined to be the supremum of all $r\geq 0$ such that, for every homogeneous pluriharmonic polynomials $P:\Omega \rightarrow \mathcal{B}(\mathcal{H})$ of the form \eqref{e-1.3-aa}, the following inequality 
\begin{equation} \label{e-snd-3.2}
	 \sum_{|\alpha|=m} \sup_{z \in r\Omega}\, (\norm{U(a_{\alpha})}^p_{Y} + \norm{U(b_{\alpha})}^p_{Y})|z^\alpha|^p \leq \lambda^p\,\norm{f}^p_{\Omega,\mathcal{B}(\mathcal{H})} 
\end{equation}
holds. Here, $\norm{f}_{\Omega,\mathcal{B}(\mathcal{H})}:=\sup_{z \in \Omega}\,\norm{f(z)}_{\mathcal{B}(\mathcal{H})}$. For short, write $L^m(\Omega, p,U):=L^m_{1}(\Omega, p,U)$, $L^m_{\lambda}(\Omega, p,X):=L^m_{\lambda}(\Omega, p,U)$ whenever $U=I:X\rightarrow X$, $L^m(\Omega, p,X):=L^m_{1}(\Omega, p,X)$, $L^m_{\lambda}(\Omega, p):=L^m_{\lambda}(\Omega, p,\mathbb{C})$, and $L^m(\Omega,p):=L^m_1(\Omega,p)$. It is straightforward to verify that $L^m_{\lambda}(\Omega, p,U)=\lambda^{1/m}\,L^m(\Omega, p,U)$.
\par 
In a similar manner, one can define the constant $B^m_{\lambda}(\Omega, p,U)$ for the class $\mathcal{P}(^m \Omega,W)$ {\it i.e.}, for any Banach space $W$ valued $m$-homogeneous holomorphic polynomial $Q(z)=\sum_{|\alpha|=m}c_{\alpha}\, z^{\alpha}$. Clearly, $B^m_{\lambda}(\Omega, p,U)=\lambda^{1/m}\,B^m(\Omega, p,U)$. On the other hand, the first Bohr radii constants for holomorhic and pluriharmonic functions, denoted respectively by $B^m_{\lambda}(\Omega, p,U)$ and $L^m_{\lambda}(\Omega, p,U)$, have already been considered in \cite{Himadri-local-Banach-1}. It is clear that, $B^m_{\lambda}(\Omega, p,U) \leq K^m_{\lambda}(\Omega, p,U)$ and $L^m_{\lambda}(\Omega, p,U) \leq R^m_{\lambda}(\Omega, p,U)$, with equality holding whenever $\Omega=\mathbb{D}^n$. Morever, it is worth noting that $B_{\lambda}(\Omega, p,U) \leq B^m_{\lambda}(\Omega, p,U)$, $K_{\lambda}(\Omega, p,U) \leq K^m_{\lambda}(\Omega, p,U)$, $R_{\lambda}(\Omega, p,U) \leq R^m_{\lambda}(\Omega, p,U)$, and $L_{\lambda}(\Omega, p,U) \leq L^m_{\lambda}(\Omega, p,U)$. Finally, in light of the Bohnenblust–Hille inequality, it is known that the constant $K^m(\mathbb{D}^n,1)=B^m(\mathbb{D}^n,1)$ is hypercontractive; see \cite{defant-2011}.
\vspace{2mm}

The following result due to \cite{Himadri-local-Banach-1}, is a coefficient-type Schwarz–Pick lemma for pluriharmonic functions defined on general complete Reinhardt domains. This lemma constitutes a fundamental ingredient in the proofs of several main results of this article.
\begin{lem} \cite{Himadri-local-Banach-1}\label{lem-3.1} 
	Let $f \in \mathcal{PH}(Z,X)$ be of the form \eqref{e-1.3-a}. Then for all $|\alpha|=m\geq 1$, we have $\norm{\sum_{|\alpha|=m}(a_{\alpha} \pm b_{\alpha})\, z^{\alpha}}_{B_{Z},X} \leq 4 \,\norm{\norm{f}_{B_{Z},X}\,I-\real(a_{0})}$. Moreover, if $B_{Z}=B_{\ell^n_{q}}$ then $\norm{a_{\alpha}+b_{\alpha}}	\leq \frac{4}{\pi} \, \rho_{\alpha}\, \norm{\sum_{|\alpha|=m}(a_{\alpha}+b_{\alpha})\, z^{\alpha}}_{B_{Z},X}$,  $\norm{a_{\alpha}-b_{\alpha}}	\leq \frac{4}{\pi}\,\rho_{\alpha}\, \norm{\sum_{|\alpha|=m}(a_{\alpha}-b_{\alpha})\, z^{\alpha}}_{B_{Z},X}$, where $\rho_{\alpha}:=\left(\frac{|\alpha|^{|\alpha|}}{\alpha^{\alpha}}\right)^{1/q}$.
	\comment{\begin{enumerate}
			\item $\norm{\sum_{|\alpha|=m}(a_{\alpha}+b_{\alpha})\, z^{\alpha}}_{B_{Z},X} \leq 4 \,\gamma_{0} \norm{f}_{B_{Z},X}$,
			$\norm{\sum_{|\alpha|=m}(a_{\alpha}-b_{\alpha})\, z^{\alpha}}_{B_{Z},X} \leq 4 \, \gamma_{0}\, \norm{f}_{B_{Z},X}$;
			\item $\norm{a_{\alpha}+b_{\alpha}}	\leq \frac{4}{\pi} \, \rho_{\alpha}\, \norm{\sum_{|\alpha|=m}(a_{\alpha}+b_{\alpha})\, z^{\alpha}}_{B_{Z},X}$,  $\norm{a_{\alpha}-b_{\alpha}}	\leq \frac{4}{\pi}\,\rho_{\alpha}\, \norm{\sum_{|\alpha|=m}(a_{\alpha}-b_{\alpha})\, z^{\alpha}}_{B_{Z},X}$,
		\end{enumerate}
		where $\gamma_{0}:=\norm{I-\real(a_{0})}$ and $\rho_{\alpha}:=\left(\frac{|\alpha|^{|\alpha|}}{\alpha^{\alpha}}\right)^{\frac{1}{q}}$.}
\end{lem}
The following result provides a fundamental connection between the Bohr radius for pluriharmonic functions and the Bohr radius for homogeneous pluriharmonic functions. It plays a crucial role in the proofs of our main results. This lemma may be viewed as the pluriharmonic analogue of \cite[Lemma 3.2]{defant-2012}, where the authors focused only on domain $\mathbb{D}^n$. The proof follows essentially the same lines as that of \cite[Lemma 3.2]{defant-2012}, and hence the details are omitted.
\begin{lem} \label{lem-3.3}
	Let $X=\mathcal{B}(\mathcal{H})$ and $Y$ any complex Banach space, and $U: X \rightarrow Y$ be a non-null bounded linear operator such that $\norm{U}< \lambda$. Then, for all $p\in[1,\infty)$, $\lambda>1$, and $n \in \mathbb{N}$, we have
	\begin{enumerate}
		\item $\left(\frac{\lambda^p - \norm{U}^p}{2\lambda^p - \norm{U}^p}\right)^{\frac{1}{p}}\, \inf_{m\in \mathbb{N}}\, L^{m}_{\lambda}(B_{Z}, p,U) \leq L_{\lambda}(B_{Z}, p,U) \leq \inf_{m\in \mathbb{N}}\, L^{m}_{\lambda}(B_{Z}, p,U)$
		
		\item $\left(\frac{\lambda^p - \norm{U}^p}{\lambda^p - \norm{U}^p +1}\right)^{\frac{1}{p}}\, \inf_{m\in \mathbb{N}}\, L^{m}(B_{Z}, p,U) \leq L_{\lambda}(B_{Z}, p,U) \leq \lambda \, \inf_{m\in \mathbb{N}}\, L^{m}(B_{Z}, p,U)$.
	\end{enumerate}
\end{lem}
\begin{rem} \label{rem-4.1}
	Analogous results of Lemma \ref{lem-3.3} hold for $B^m_{\lambda}(B_{Z}, p,U)$ and $B_{\lambda}(B_{Z}, p,U)$.
\end{rem}
To simplify computations in the proofs of the remaining theorems, we may assume without loss of generality that every pluriharmonic function $f \in \mathcal{PH}(B_{Z},X)$ satisfies $\norm{f}_{B_{Z},X} \leq 1$, where $X=\mathcal{B}(\mathcal{H})$. That is to say, we restrict our attention to the class $$\mathcal{BPH}(B_{Z}, X):=\{f: f \in \mathcal{PH}(B_{Z},X)\, \mbox{ with}\, \norm{f}_{B_{Z},X} \leq 1\}.$$
This is justified by a simple normalization argument. Indeed, if $f\in \mathcal{PH}(B_{Z},X)$ satisfies $\norm{f}_{B_{Z},X} > 1$ and admits the representation \eqref{e-1.3-a}, then the function $F(z):=f(z)/\norm{f}_{B_{Z},X}$
belongs to $\mathcal{PH}(B_{Z},X)$ and satisfies $\norm{F}_{B_{Z},X} =1$. In this normalization the coefficients $a_{\alpha}$ and $b_{\alpha}$ in \eqref{e-1.3-a} are replaced by $a_{\alpha}/\norm{f}_{B_{Z},X}$ and $b_{\alpha}/\norm{f}_{B_{Z},X}$, respectively. By the same reasoning, we also consider the class of normalized $m$-homogeneous polynomials,
$$\mathcal{BPH}(^mB_{Z},X):=\{P: P \in \mathcal{PH}(^mB_{Z},X)\,\, \mbox{with}\,\, \norm{P}_{B_{Z},X} \leq 1\}.
$$
\subsection{Proof of lower bounds of Theorem \ref{thm-1.2}} 
	We want to first obtain the desired lower bound of $R^{m}_{\lambda}(B_{Z}, p,X)$ when $X$ is finite dimensional. For all $m\geq1$, let $P(z)=\sum_{|\alpha|=m} a_{\alpha}\, z^{\alpha} + \sum_{|\beta|=m} b^{*}_{\alpha}\, \bar{z}^{\alpha} \in \mathcal{BPH}(^mB_{Z},X)$. It is known from \cite[Lemma $3.5$]{defant-2004} that for $1 \le q \le \infty$ and $|\alpha| = m$,
	\[
	\sup_{z \in B_{\ell_q^n}} |z^\alpha|
	=
	\left(\frac{\alpha^\alpha}{m^m}\right)^{\frac{1}{q}}.
	\]
	
	Now let $B_Z$ denote the unit ball of a complex Banach space $Z = (\mathbb{C}^n, \|\cdot\|)$.
	We want to find an upper bound for $\sup_{z \in B_Z} |z^\alpha|$.
	Since $B_Z \subseteq \|\mathrm{Id} : Z \to \ell_q^n\|\, B_{\ell_q^n} $, for any $z \in B_Z$, we have  
	\[
	|z^\alpha|
	\le
	\sup_{w \in B_{\ell_q^n}} |w^\alpha|
	\,
	\|\mathrm{Id} : Z \to \ell_q^n\|^{\,m}, \,\,\,\,\,\, |\alpha|=m.
	\]
	Hence,
	\begin{equation} \label{e-z-mod}
		\sup_{z \in B_Z} |z^\alpha|
		\le
		\left(\frac{\alpha^\alpha}{m^m}\right)^{\frac{1}{q}}
		\,
		\|\mathrm{Id} : Z \to \ell_q^n\|^{\,m}
	\end{equation}
for $|\alpha|=m$. Note that $\left(\frac{\alpha^\alpha}{m^m}\right)^{\frac{1}{q}} \leq 1$. Then for $r \in [0,1)$, \eqref{e-z-mod} yields
	\begin{equation} \label{snd-e-4.6}
		\sum_{|\alpha|=m} \sup_{z \in r \, B_{Z}} (\norm{a_{\alpha}}^p+\norm{b_{\alpha}}^p)|z|^{p\alpha} \leq r^{pm}\, \norm{\mathrm{Id}: Z \rightarrow \ell^n_q}^{pm}\,\sum_{|\alpha|=m}  (\norm{a_{\alpha}}^p+\norm{b_{\alpha}}^p).
	\end{equation}
	We now need to estimate the term $\sum_{|\alpha|=m}  (\norm{a_{\alpha}}^p+\norm{b_{\alpha}}^p)$. To do so, we consider several cases depending on the values of $p$. More precisely, the proof is presented in three separate cases, depending on the values of $p$: $p=1$, $p\geq 2$, and $1<p<2$. Let $H$ and $G$ be the $m$-homogeneous polynomials as defined by 
	\[
	H(z)=\sum_{|\alpha|=m}\frac{a_\alpha+b_\alpha}{2}\,z^\alpha,
	\qquad
	G(z)=\sum_{|\alpha|=m}\frac{a_\alpha-b_\alpha}{2}\,z^\alpha.
	\]
	 Set $s:=\norm{\mathrm{Id}: Z \rightarrow \ell^n_q}$ and $t:=\norm{\mathrm{Id}:  \ell^n_q\rightarrow Z}$. \\
	\underline{{\bf Case $p=1$:}} 
	For each $\phi \in B_{X^{*}}$, we now consider the holomorphic polynomial $\tilde{H}:=\phi \circ H:\mathbb{C}^n \rightarrow \mathbb{C}$ defined by
	$\tilde{H}(z)=\phi(H(z))= \sum_{|\alpha|=m} \phi\left(\frac{a_{\alpha} + b_{\alpha}}{2}\right) z^{\alpha}$.
	Clearly, $|\tilde{H}(z)| \leq \norm{H(z)}$ for each $z$. Since $X$ is finite dimensional, the identity operator $I_{X}$ is $1$-summing. Then 
	\begin{align} \label{e-1.19}
		\left(\sum_{|\alpha|=m} \norm{\left(\frac{a_{\alpha} + b_{\alpha}}{2}\right)}\right) \left(\frac{1}{t \, n^{\frac{1}{q}}}\right)^m \nonumber
		& = \sum_{|\alpha|=m} \norm{\left(\frac{a_{\alpha} + b_{\alpha}}{2}\right)} \left(\frac{1}{t \, n^{\frac{1}{q}}}\right)^{\alpha} \\ \nonumber
		&\leq \pi_{1}(I_{X}) \sup_{\phi \in B_{X^{*}}} \left(\sum_{|\alpha|=m}\left|\phi\left(\frac{a_{\alpha} + b_{\alpha}}{2}\right) \right|\right)\, \left(\frac{1}{t \, n^{\frac{1}{q}}}\right)^m \\ \nonumber
		& = \pi_{1}(I_{X}) \sup_{\phi \in B_{X^{*}}} \left(\sum_{|\alpha|=m}\left|\phi\left(\frac{a_{\alpha} + b_{\alpha}}{2}\right) \right|\, \left(\frac{1}{t \, n^{\frac{1}{q}}}\right)^{\alpha} \right) \\ \nonumber
		& \leq \pi_{1}(I_{X}) \sup_{\phi \in B_{X^{*}}} \sup_{z \in B_{Z}} \left|\sum_{|\alpha|=m}\left|\phi\left(\frac{a_{\alpha} + b_{\alpha}}{2}\right) \right|z^{\alpha}\right| \\ \nonumber
		& = \pi_{1}(I_{X}) \sup_{\phi \in B_{X^{*}}} \sup_{z \in B_{Z}} \left|\sum_{|\alpha|=m} \varepsilon_{\alpha}\phi\left(\frac{a_{\alpha} + b_{\alpha}}{2}\right) z^{\alpha}\right| \\ \nonumber
		& \leq \pi_{1}(I_{X})\, \chi_{M}(\mathcal{P}(^m Z)) \, \sup_{\phi \in B_{X^{*}}} \sup_{z \in B_{Z}} \left|\sum_{|\alpha|=m}\phi\left(\frac{a_{\alpha} + b_{\alpha}}{2}\right) z^{\alpha}\right| \\ \nonumber
		& = \pi_{1}(I_{X})\, \chi_{M}(\mathcal{P}(^m Z)) \, \sup_{\phi \in B_{X^{*}}} \sup_{z \in B_{Z}} \left|\phi \left(\sum_{|\alpha|=m}\left(\frac{a_{\alpha} + b_{\alpha}}{2}\right) z^{\alpha}\right)\right| \\ 
		& = \pi_{1}(I_{X})\, \chi_{M}(\mathcal{P}(^mZ)) \,   \norm{H}_{B_{Z},X}, 
	\end{align}
	where $\varepsilon_{\alpha}=\frac{\left|\phi\left(\frac{a_{\alpha} + b_{\alpha}}{2}\right)\right|}{\phi\left(\frac{a_{\alpha} + b_{\alpha}}{2}\right)}$ in the last fourth inequality. The last third inequality follows from the definition of $\chi_{M}(\mathcal{P}(^m Z))$. From \eqref{e-1.19}, we have
	\begin{equation} \label{snd-e-4.8}
		\sum_{|\alpha|=m} \norm{\left(\frac{a_{\alpha} + b_{\alpha}}{2}\right)} \leq \pi_{1}(I_{X})\, \chi_{M}(\mathcal{P}(^mZ)) \, \left(t\, n^{\frac{1}{q}}\right)^m\, \norm{H}_{B_{Z},X}.
	\end{equation}
	
	By considering the holomorphic polynomial $\tilde{G}:=\phi \circ G:\mathbb{C}^n \rightarrow \mathbb{C}$, proceeding as above we obtain 
	\begin{equation} \label{snd-e-4.9}
		\sum_{|\alpha|=m} \norm{\left(\frac{a_{\alpha} - b_{\alpha}}{2}\right)} \leq \pi_{1}(I_{X})\, \chi_{M}(\mathcal{P}(^mZ)) \, \left(t\, n^{\frac{1}{q}}\right)^m\, \norm{G}_{B_{Z},X}.
	\end{equation} By making use of \eqref{snd-e-4.8} and \eqref{snd-e-4.9}, and then from Lemma \ref{lem-3.1}, a simple computation shows that 
	$$\sum_{|\alpha|=m} \norm{a_{\alpha}} \leq \pi_{1}(I_{X})\, \chi_{M}(\mathcal{P}(^m Z))\, \left(t\, n^{\frac{1}{q}}\right)^m \, \left(\norm{H}_{B_{Z},X}+\norm{G}_{B_{Z},X}\right)  \leq 4\, \pi_{1}(I_{X})\, \chi_{M}(\mathcal{P}(^m Z))\, \left(t\, n^{\frac{1}{q}}\right)^m.
	$$
	Similarly, we get 
	$$
	\sum_{|\alpha|=m} \norm{b_{\alpha}} \leq 4\, \pi_{1}(I_{X})\, \chi_{M}(\mathcal{P}(^m Z))\, \left(t\, n^{\frac{1}{q}}\right)^m.
	$$
	In view of \eqref{snd-e-4.6} and last two inequalities, we obtain 
	$$\sum_{|\alpha|=m} \sup_{z \in r \, B_{Z}} (\norm{a_{\alpha}}+\norm{b_{\alpha}})|z|^{\alpha} \leq 8\,r^m\, \pi_{1}(I_{X})\, \chi_{M}(\mathcal{P}(^m Z))\, \left(s\,t\, n^{\frac{1}{q}}\right)^m \leq \lambda
	$$
	provided
	$$
	r \leq \frac{\lambda^{\frac{1}{m}}}{8^{\frac{1}{m}}\, (\pi_{1}(I_{X}))^{\frac{1}{m}}\, (\chi_{M}(\mathcal{P}(^m Z)))^{\frac{1}{m}}\, s\,t\, n^{\frac{1}{q}}}.
	$$
	This yields
	\begin{equation} \label{snd-e-4.10}
		L^{m}_{\lambda}(B_{Z}, p,X)  \geq E_{1}(X)\, \frac{\lambda^{\frac{1}{m}}}{ (\chi_{M}(\mathcal{P}(^m Z)))^{\frac{1}{m}}\, s\,t\, n^{\frac{1}{q}}},
	\end{equation}
	where $E_{1}(X)$ is a constant depending on $X$. Thanks to the the following estimates due to \cite[(4.5), (4.6), pp. 187]{defant-2003}:
	\begin{equation} \label{snd-e-4.10-a}
		\chi_{M}(\mathcal{P}(^m Z)) \leq n^{\frac{m}{2}}\, \norm{\mathrm{Id}:\ell^n_2 \rightarrow Z}^m\,\, \mbox{and}\,\, \chi_{M}(\mathcal{P}(^m Z)) \leq \frac{m^m}{m!}\, \norm{\mathrm{Id}: Z \rightarrow \ell^n_1}^m,
	\end{equation}
	from \eqref{snd-e-4.10}, we deduce that 
	\begin{equation*}
		L^{m}_{\lambda}(B_{Z}, p,X) \geq E_{1}(X)\,  \,\frac{\lambda^{\frac{1}{m}}}{s\,t\, n^{\frac{1}{q}}}\, \max \left\{\frac{1}{\sqrt{n}\, \norm{\mathrm{Id}: \ell^n_{2} \rightarrow Z}}, \, \frac{1}{\left(\frac{m^m}{m!}\right)^{1/m}\, \norm{\mathrm{Id}: Z \rightarrow \ell^n_{1}}}\right\},
	\end{equation*}
	and hence the desired lower bound of $L_{\lambda}(B_{Z}, p,X)$ follows from Lemma \ref{lem-3.3} (1) and using the fact $\sup _{m \in \mathbb{N}} (m^m/m!)^{1/m} =e$. This completes the proof for $p=1$.
	\\ [1mm]
	\underline{{\bf Case $p\geq 2$:}} Let $p\geq 2$. By Lemma \ref{lem-3.1} and keeping in mind $\norm{P}_{B_Z,X} \leq 1$, one sees easily that
	\begin{equation*}
		\norm{\left(\frac{a_{\alpha} + b_{\alpha}}{2}\right)} \leq  \norm{H}_{B_Z,X} \leq 2\, \norm{P}_{B_Z,X} \leq 2
	\end{equation*}
	and so, 
	$$\sum_{|\alpha|=m} \norm{\left(\frac{a_{\alpha} + b_{\alpha}}{2}\right)}^p \leq 2^{p-2}\, \sum_{|\alpha|=m} \norm{\left(\frac{a_{\alpha} + b_{\alpha}}{2}\right)}^2.
	$$
	Then, this inequality, together with the fact $\norm{(a_{\alpha}+b_{\alpha})/2}=\sup_{\phi \in B_{X^{*}}}|\phi((a_{\alpha}+b_{\alpha})/2)|$and the following estimate due to \cite[p. 187]{defant-2003}
	\begin{equation*}
		\left(\sum_{|\alpha|=m} \left|\phi\left(\frac{a_{\alpha} + b_{\alpha}}{2}\right)\right|^2\right)^{\frac{1}{2}} \leq \norm{\mathrm{Id}:\ell^n _2 \rightarrow Z}^{m} \norm{H}_{B_{Z},X}
	\end{equation*}
	yields
	\begin{equation} \label{snd-e-4.11}
		\sum_{|\alpha|=m} \norm{\left(\frac{a_{\alpha} + b_{\alpha}}{2}\right)}^p\leq 2^{p-2}\, \norm{\mathrm{Id}:\ell^n _2 \rightarrow Z}^{2m} \norm{H}^2_{B_{Z},X}.	
	\end{equation}
	Similarly, we have 
	\begin{equation} \label{snd-e-4.12}
		\sum_{|\alpha|=m} \norm{\left(\frac{a_{\alpha} - b_{\alpha}}{2}\right)}^p\leq 2^{p-2}\, \norm{\mathrm{Id}:\ell^n _2 \rightarrow Z}^{2m} \norm{G}^2_{B_{Z},X}.
	\end{equation}
	\comment{\begin{align} \label{e-1.24}
			\sum_{|\alpha|=m} \norm{\left(\frac{a_{\alpha} + b_{\alpha}}{2}\right)z^{\alpha}}^p \nonumber
			& \leq \sum_{|\alpha|=m} \norm{\left(\frac{a_{\alpha} + b_{\alpha}}{2}\right)z^{\alpha}}^2 \\ \nonumber
			& \leq  \sum_{|\alpha|=m} \norm{\left(\frac{a_{\alpha} + b_{\alpha}}{2}\right)}^2 \\ 
			& \leq  \norm{I:\ell^n _2 \rightarrow Z}^{2m} \norm{H}^2_{B_{Z},X}.
		\end{align}
		Similarly,
		\begin{equation} \label{e-1.25}
			\sum_{|\alpha|=m} \norm{\left(\frac{a_{\alpha} - b_{\alpha}}{2}\right)z^{\alpha}}^p \leq \norm{I:\ell^n _2 \rightarrow Z}^{2m} \norm{G}^2_{B_{Z},X}.
	\end{equation}}
	Now, by using Lemma \ref{lem-3.1} again, \eqref{snd-e-4.11} and \eqref{snd-e-4.12} taken together yield 
	$$\left(\sum_{|\alpha|=m} \norm{a_{\alpha}}^p\right)^{\frac{1}{p}}  \leq 2 \norm{\mathrm{Id}:\ell^n _2 \rightarrow Z}^{\frac{2m}{p}}  2^{\frac{2}{p}}.$$
	Moreover, one can easily see that this inequality remains valid if the left hand quantity replaced by $\left(\sum_{|\alpha|=m} \norm{b_{\alpha}}^p\right)^{1/p} $.
	We now use \eqref{snd-e-4.6} and the last two inequalities to prove 
	$$
	\sum_{|\alpha|=m} \sup_{z \in r \, B_{Z}} (\norm{a_{\alpha}}^p+\norm{b_{\alpha}}^p)|z|^{p\alpha} \leq 2^{p+3}  \,r^{pm}\, s^{pm}\, \norm{\mathrm{Id}:\ell^n _2 \rightarrow Z}^{2m} \leq \lambda^p
	$$
	provided 
	\begin{equation*}
		r \leq \frac{\lambda^{\frac{1}{m}}}{2^{\frac{p+3}{pm}}\,s\, \norm{\mathrm{Id}:\ell^n _2 \rightarrow Z}^{\frac{2}{p}}}.
	\end{equation*}
	Consequently, 
	$$L^{m}_{\lambda}(B_{Z}, p,X) \geq \frac{\lambda^{\frac{1}{m}}}{2^{\frac{p+3}{pm}}\,s\, \norm{\mathrm{Id}:\ell^n _2 \rightarrow Z}^{\frac{2}{p}}}.$$
	Finally, the desired lower bound of $R_{\lambda}(B_{Z}, p,X)$ follows from Lemma \ref{lem-3.3} (1).\\
	\underline{{\bf Case $1<p<2$:}} For any $1<p<2$, we shall make use of the estimates obtained in the above two cases. We now use H\"{o}lder's inequality to get 
	\begin{equation*}
		\sum_{|\alpha|=m} \norm{\left(\frac{a_{\alpha} + b_{\alpha}}{2}\right)}^p \leq  \left(\sum_{|\alpha|=m} \norm{\left(\frac{a_{\alpha} + b_{\alpha}}{2}\right)}\right)^{2-p} \left(\sum_{|\alpha|=m} \norm{\left(\frac{a_{\alpha} + b_{\alpha}}{2}\right)}^2\right)^{p-1}.
	\end{equation*}
	By virtue of \eqref{snd-e-4.8}, \eqref{snd-e-4.11}, and Lemma \ref{lem-3.1}, we obtain
	$$
	\sum_{|\alpha|=m} \norm{\left(\frac{a_{\alpha} + b_{\alpha}}{2}\right)}^p \leq 2^{p}\,(\pi_{1}(I_{X}))^{2-p}\, \left(t\, n^{\frac{1}{q}}\right)^{(2-p)m} (\chi_{M}(\mathcal{P}(^mZ)))^{2-p} \norm{\mathrm{Id}:\ell^n _2 \rightarrow Z}^{2(p-1) m}.$$
	Similarly, we have
	$$
	\sum_{|\alpha|=m} \norm{\left(\frac{a_{\alpha} - b_{\alpha}}{2}\right)}^p \leq 2^{p}\,(\pi_{1}(I_{X}))^{2-p}\, \left(t\, n^{\frac{1}{q}}\right)^{(2-p)m} (\chi_{M}(\mathcal{P}(^mZ)))^{2-p} \norm{\mathrm{Id}:\ell^n _2 \rightarrow Z}^{2(p-1) m}.$$
	\comment{\begin{align*}
			&\sum_{|\alpha|=m} \norm{\left(\frac{a_{\alpha} + b_{\alpha}}{2}\right)z^{\alpha}}^p \\ \nonumber
			&=\sum_{|\alpha|=m} \norm{\left(\frac{a_{\alpha} + b_{\alpha}}{2}\right)z^{\alpha}}^{p(1-\theta)}\, \norm{\left(\frac{a_{\alpha} + b_{\alpha}}{2}\right)z^{\alpha}}^{p\theta} \\ \nonumber
			& \leq \left(\sum_{|\alpha|=m} \norm{\left(\frac{a_{\alpha} + b_{\alpha}}{2}\right)z^{\alpha}}\right)^{p(1-\theta)} \left(\sum_{|\alpha|=m} \norm{\left(\frac{a_{\alpha} + b_{\alpha}}{2}\right)z^{\alpha}}^2\right)^{\frac{p\theta}{2}} \\ \nonumber
			& \leq (\pi_{1}(I_{X}))^{p(1-\theta)}\, (\chi_{M}(\mathcal{P}(^mZ)))^{p(1-\theta)} \norm{I:\ell^n _2 \rightarrow X}^{p \theta m} \norm{H}^{p}_{B_{Z},X} \\ \nonumber
			& \leq 4^{p}\,(\pi_{1}(I_{X}))^{p(1-\theta)}\, (\chi_{M}(\mathcal{P}(^mZ)))^{p(1-\theta)} \norm{I:\ell^n _2 \rightarrow X}^{p \theta m} \, \norm{I-\real(a_{0})}^{p}\, \norm{f}^p_{B_{Z},X}.
	\end{align*}}
	Now, use these two inequalities to prove
	$$\left(\sum_{|\alpha|=m} \norm{a_{\alpha}}^p\right)^{\frac{1}{p}} \leq 
	4\,(\pi_{1}(I_{X}))^{\frac{(2-p)}{p}}\, \left(t\, n^{\frac{1}{q}}\right)^{\frac{(2-p)m}{p}} (\chi_{M}(\mathcal{P}(^mZ,Y)))^{\frac{(2-p)}{p}} \norm{\mathrm{Id}:\ell^n _2 \rightarrow Z}^{ \frac{2(p-1)m}{p}}.$$
	This inequality remains valid if the left-hand side is replaced by $\left(\sum_{|\alpha|=m} \norm{b_{\alpha}}^p\right)^{1/p}$. These inequalities and \eqref{snd-e-4.6} taken together yield 
	\begin{align*}
		&\sum_{|\alpha|=m} \sup_{z \in r \, B_{Z}} (\norm{a_{\alpha}}^p+\norm{b_{\alpha}}^p)|z|^{p\alpha}\\
		& \leq 2^{2p+1}(\pi_{1}(I_{X}))^{2-p}\,r^{pm}\, s^{pm}\, \left(t\, n^{\frac{1}{q}}\right)^{(2-p)m} (\chi_{M}(\mathcal{P}(^mZ,Y)))^{2-p} \norm{\mathrm{Id}:\ell^n _2 \rightarrow Z}^{2(p-1) m},
	\end{align*}
	and so, 
	$$L^{m}_{\lambda}(B_{Z}, p,X) \geq E_{2}(X)\,  \,
	\frac{\lambda^{\frac{1}{m}}}{s\,\left(t\, n^{\frac{1}{q}}\right)^{\frac{(2-p)}{p}}(\chi_{M}(\mathcal{P}(^mZ)))^{\frac{(2-p)}{pm}} \norm{\mathrm{Id}:\ell^n _2 \rightarrow Z}^{\frac{2(p-1)}{p}}}
	,$$ where $E_{2}(X)$ is a constant depending on $X$ and $p$. Then, letting $\theta = (2(p-1))/p$, the estimates in \eqref{snd-e-4.10-a} show
	\begin{equation*}
		L^{m}_{\lambda}(B_{Z}, p,X) \geq  \frac{E_{2}(X)\,\lambda^{\frac{1}{m}}}{s\,\left(t\, n^{\frac{1}{q}}\right)^{\frac{(2-p)}{p}}}\max \left\{\frac{1}{(\sqrt{n})^{1-\theta}\, \norm{\mathrm{Id}: \ell^n_{2} \rightarrow Z}}, \, \frac{\norm{\mathrm{Id}:\ell^n _2 \rightarrow Z}^{-\theta}}{\left(\frac{m^m}{m!}\right)^{(1-\theta)/m}\, \norm{\mathrm{Id}: Z \rightarrow \ell^n_{1}}^{1-\theta}}\right\},
	\end{equation*}
	and hence the desired lower bound of $L_{\lambda}(B_{Z}, p,X)$ follows from Lemma \ref{lem-3.3} (1) and using the fact $\sup _{m \in \mathbb{N}} (m^m/m!)^{1/m} =e$. 

\subsection{Proof of upper bounds of Theorem \ref{thm-1.2}} 	
	We now proceed to derive the desired upper bound for $L_{\lambda}(B_{Z}, p,X)$. More generally, we establish this bound for any simply connected complete Reinhardt domain $\Omega$. Our approach is inspired by some of the arguments used in \cite[Theorem 3.3]{defant-2004}, which were developed for complex valued holomorphic functions; here, however, we work in the setting of operator valued pluriharmonic functions.
	For $\varepsilon_{\alpha} \in \{-1,1\}$ and $I$ the identity operator on $\mathcal{H}$, we consider the $m$-homogeneous polynomial $f:\Omega \rightarrow \mathcal{B}(\mathcal{H})$ defined as 
	\begin{equation*}
		f(z)=\sum_{|\alpha|=m} \varepsilon_{\alpha}\, \frac{m!}{\alpha!}z^{\alpha}\, I +  \sum_{|\alpha|=m} \varepsilon_{\alpha}\,\frac{m!}{\alpha!}\, \bar{z}^{\alpha} I, \,\, z \in \Omega.
	\end{equation*}
	Then, by definition of the second Bohr radius for every $0<\epsilon<1$, we have
	\begin{equation*}
		\sum_{|\alpha|=m} \sup_{z \in (1-\epsilon)L\,\Omega} \left(\left(\frac{m!}{\alpha!}\right)^p+ \left(\frac{m!}{\alpha!}\right)^p\right)\, |z|^{p\alpha}  =2\, \sum_{|\alpha|=m}  \left(\frac{m!}{\alpha!}\right)^p\, \sup_{z \in (1-\epsilon)L\,\Omega} |z|^{p\alpha} \leq \lambda^p \,\norm{f}^p_{\Omega,X},
	\end{equation*}
	where $L:=L_{\lambda}(\Omega, p,X)$ for short. Now, write $b_m(\Omega):=\left(\inf_{|\alpha|=m}\sup_{z \in \Omega} |z^{\alpha}|\right)^{-\frac{1}{m}}$, $m \in \mathbb{N}$, and use the facts $m!/\alpha! \geq 1$ and $\sum_{|\alpha|=m} \frac{m!}{\alpha!}=n^m$ to prove
	\begin{equation*}
		2\, n^m \, L^{pm}\, (b_m(\Omega))^{-pm} \leq \lambda^p \,\norm{f}^p_{\Omega,X}. 
	\end{equation*}
	Consequently,
	\begin{equation} \label{snd-e-4.14}
		L \leq \frac{\lambda^{\frac{1}{m}}\,b_m(\Omega)\, \norm{f}^{\frac{1}{m}}_{\Omega,X} }{2^{\frac{1}{pm}}\, n^{\frac{1}{p}}}.
	\end{equation}
	In order to obtain an upper bound of $L$, it is enough to estimate the quantity $\norm{f}_{\Omega,X}$. Notice that $\norm{f}_{\Omega,X} \leq 2\, \norm{\sum_{|\alpha|=m} \varepsilon_{\alpha}\, \frac{m!}{\alpha!}z^{\alpha}}_{\Omega,\mathbb{C}}$. We shall now make use of the following estimate due to \cite[Theorem 2.3]{defant-2004} (see also \cite[p. 21, equ. 21]{defant-2007}, \cite[Theorem 3.1]{defant-2003}):
	\begin{equation*}
		\norm{\sum_{|\alpha|=m} \varepsilon_{\alpha}\,\frac{m!}{\alpha!}\, z^{\alpha} }_{\Omega,\mathbb{C}} \leq \sqrt{m^3\, 2^{3m-1}\, \log\,n}\, \sup _{|\alpha|=m}\left\{\sqrt{\frac{m!}{\alpha!}}\right\} \left(\sup_{z \in \Omega}\sum_{i=1}^{n} |z_i|\right) \left(\sup_{z \in \Omega}\left(\sum_{i=1}^{n} |z_i|^2\right)^{1/2}\right)^{m-1}.
	\end{equation*}
	Recalling the fact $S(\Omega,B_{\ell^n_q})=\sup_{z \in \Omega} \norm{z}_{q}$, $1 \leq q \leq \infty$, the above estimate and \eqref{snd-e-4.14} taken together yield
	\begin{align} \label{snd-e-4.15}
		L & \leq 2^{\frac{1}{m}\left(1-\frac{1}{p}\right)}\,\lambda^{\frac{1}{m}}\,b_m(\Omega)\, \left(m^3\,m!\, 2^{3m-1}\, \log\,n\right)^{\frac{1}{2m}}\, (S(\Omega,B_{\ell^n_1}))^{\frac{1}{m}}\, \frac{(S(\Omega,B_{\ell^n_2}))^{\frac{m-1}{m}}}{n^{\frac{1}{p}}} \nonumber\\[2mm]
		& = 2^{\frac{1}{m}\left(1-\frac{1}{p}\right)}\,\lambda^{\frac{1}{m}}\,b_m(\Omega)\, \left(m^3\,m!\, 2^{3m-1}\, \log\,n\right)^{\frac{1}{2m}}\, \left(\frac{S(\Omega,B_{\ell^n_1})}{S(\Omega,B_{\ell^n_2})}\right)^{\frac{1}{m}}\, \frac{S(\Omega,B_{\ell^n_2})}{n^{\frac{1}{p}}} \nonumber\\[2mm]
		& \leq 2^{\frac{1}{m}\left(1-\frac{1}{p}\right)}\,\lambda^{\frac{1}{m}}\,b_m(\Omega)\, \left(m^3\,m!\, 2^{3m-1}\,n\, \log\,n\right)^{\frac{1}{2m}}\,\frac{S(\Omega,B_{\ell^n_2})}{n^{\frac{1}{p}}}, 
	\end{align}
	where the last inequality follows from the fact that $\frac{S(\Omega,B_{\ell^n_1})}{S(\Omega,B_{\ell^n_2})} \leq \sqrt{n}$. Indeed, it is well-known that $\norm{z}_1 \leq \sqrt{n}\, \norm{z}_2 $ for $z \in \Omega$, which is again less than or equals to $\sqrt{n}\, S(\Omega,B_{\ell^n_2})$. So, $\sup_{z \in \Omega} \norm{z}_1 \leq \sqrt{n}\, S(\Omega,B_{\ell^n_2})$.
	\vspace{1mm}
	
	For the desired upper bound of $L$, it is now enough to estimate the right side quantity of \eqref{snd-e-4.15}. To do so, we first estimate the term 
	$A_{m,n}:=\left(m^3\,m!\,2^{3m-1}\,n\,\log\,n\right)^{\frac{1}{2m}}$. Stirling's formula asserts that
	$$m! \sim \sqrt{2\, \pi m} \left(\frac{m}{e}\right)^m.
	$$
	Substituting this into $A_{m,n}$, a straightforward calculation gives
	\begin{align*}
		A_{m,n} \leq C \left(\left(\frac{m}{e}\right)^m\, 2^{3m}\right)^{\frac{1}{2m}}\, m^{\frac{7}{4m}}\,(n\, \log \,n)^{\frac{1}{2m}},
	\end{align*} 
	where $C>0$ is a constant, independent both of $m$ and $n$. Since $( m)^{\frac{7}{4m}} \rightarrow 1$ as $m \rightarrow \infty$, there exists a constant $C_1>0$ such that the last inequality simplifies to
	\begin{align*}
		A_{m,n} \leq C_1 \, \sqrt{m}\,  \,(n\, \log \,n)^{\frac{1}{2m}}. 
	\end{align*}
	Now, choose $m=[\log\, n]$ for $n \geq 3$. Then
	\begin{equation*}
		(n\, \log \,n)^{\frac{1}{2m}}= \exp \left(\frac{\log \,n + \log\log \,n}{2m}\right).
	\end{equation*}
	Since $m \sim \log \, n$,
	\begin{equation*}
		\frac{\log \,n}{2m} \leq \frac{1}{2} + o(1)\,\, \mbox{and} \, \, \frac{\log\log \,n}{2\, m} \rightarrow 0\,\, \mbox{as}\,\,n \rightarrow \infty.
	\end{equation*}
	Hence
	\begin{equation*}
		(n\, \log \,n)^{\frac{1}{2m}} \leq C_2,
	\end{equation*}
	for some constant $C_2>0$, which does not depend on neither $n$ nor $m$. Finally,
	\begin{equation*}
		A_{m,n} \leq C_3 \sqrt{\log \, n},
	\end{equation*}
	with $C_3>0$ being another constant and so,
	\begin{equation*}
		L\leq C_4 \, \lambda^{\frac{1}{\log \,n}}\,b_{[\log\,n]}(\Omega)\, \,\frac{S(\Omega,B_{\ell^n_2})}{n^{\frac{1}{p}}}
	\end{equation*}
	for some constant $C_4>0$, independent both of $m$ and $n$. This constant may depend on $p$. In view of the identity $S(B_Z,B_{\ell^n_2})=\norm{\mathrm{Id}:Z \rightarrow \ell^n_2}$, the desire upper bound of $L_{\lambda}(B_{Z}, p,X)$ is immediate. This completes the proof.

	\comment{In particular, for the Minkowski spaces $Z=\ell^n _q$, we obtain the following lower bound. The upper for this obtained in the proof of Theorem \ref{thm-1.2}.
		\begin{cor} \label{cor-1.4}
			Let $Z=(\mathbb{C}^n, ||.||)$ be a Banach space such that $\chi(\{e_k\}^n_{k=1})=1$.	Let $X=\mathcal{B}(\mathcal{H})$ be finite dimensional and $\lambda>1$. Then
			\begin{equation*}
				R_{\lambda}(B_{\ell^n _q}, p,X) \geq	\begin{cases}
					E'_{3}(X)\,\left(\frac{\lambda^p -1}{2\lambda^p - 1}\right)^{\frac{1}{p}} \, \left(\frac{\log n}{n}\right)^{ \left(1- \frac{1}{\min\{q,2\}}\right)} & \text{for $p=1$}, \\[2mm]
					E'_{2} \left(\frac{\lambda^p -1}{2\lambda^p - 1}\right)^{\frac{1}{p}}\, n^{-\frac{1}{p}} & \text{for $p \geq q$}, \\[2mm]
					
					E_{4}(X)\,  \left(\frac{\lambda^p -1}{2\lambda^p - 1}\right)^{\frac{1}{p}} \,n^{-\frac{p-1}{p(q-1)}}\, \left( \frac{\log n}{n}\right)^{\left(1- \frac{1}{\min\{q,2\}}\right)\, \frac{q-p}{p(q-1)}} & \text{for $1<p< q$}. 
				\end{cases}
			\end{equation*}
			Here, $E'_{3}(X)$, $E'_{2}$, and $E_{4}(X)$ are positive constants: $E'_{3}(X)$ depends only on $X$, $E_{4}(X)$ depends on both $X$ and $p$, while $E'_{2}$ is independent of $X$ and depends only on $p$
	\end{cor}}

\section{Estimates of Second Bohr radii constants in infinite dimensional settings}
In this section, we investigate the asymptotic behaviour of the second Bohr radius constant in the infinite-dimensional setting. 
We show that this asymptotic estimate is governed by the geometric structure of the underlying Banach space $X$, more precisely, by its optimal cotype $\mathrm{Cot}(X)$.
\begin{thm} \label{thm-4.1}
	Let $X=\mathcal{B}(\mathcal{H})$ be an infinite dimensional complex Banach space  and $\lambda>1$. Let $1 \leq p,q < \infty$. Then 
	\begin{equation*}
		L_{\lambda}(B_{Z}, p,X) \geq	\begin{cases}
				E_{5}\, \dfrac{\left(\lambda^p -1\right)^{\frac{1}{p}}}{\lambda}\, \,  \frac{1}{ \norm{\mathrm{Id}: Z \rightarrow \ell^n_q}\,\|\mathrm{Id}:\ell_q^n\to Z\|\, n^{\frac{1}{q}}} & \text{for $p > q$}, \\[2mm]
			
			E_{6}\,\dfrac{\left(\lambda^p -1\right)^{\frac{1}{p}}}{\lambda}\, \,  \frac{1}{ \norm{\mathrm{Id}: Z \rightarrow \ell^n_\infty}\,\|\mathrm{Id}:\ell_\infty^n\to Z\|} & \text{for $p>\mathrm{Cot}(X)$}, \\[2mm]
			E_{7}\, \dfrac{\left(\lambda^p -1\right)^{\frac{1}{p}}}{\lambda}\,  \, \frac{ n^{\frac{1}{t}-\frac{1}{p}}}{\|\mathrm{Id}: Z\to \ell_\infty^n\|\,\|\mathrm{Id}:\ell_\infty^n\to Z\|} & \text{for $p\leq  \mathrm{Cot}(X)$},\\[2mm]
			E_{8}\, \dfrac{\left(\lambda^p -1\right)^{\frac{1}{p}}}{\lambda}\,   \frac{1}{ \norm{\mathrm{Id}: Z \rightarrow \ell^n_q}\,\|\mathrm{Id}:\ell_q^n\to Z\|\, \, n^{\frac{1}{p}}} & \text{for $p \leq q$}
		\end{cases}
	\end{equation*}
	if $X$ has a finite cotype $t$. Here $E_{5},E_{6},E_{7}$, and $E_{8}$ are strictly positive constants independent both of $X$ and $n$. 
	On the other hand, we have
	\begin{equation*}
		L_{\lambda}(B_{Z}, p,X) \leq 	
			E_9\,\lambda\, n^{\frac{1}{\mathrm{Cot}(X)}\, - \,\frac{1}{p}}\,\, \| \mathrm{Id}: Z \to \ell_\infty^n \|\, \|\mathrm{Id}: \ell_\infty^n \to Z\| 
	\end{equation*}
for some $E_9>0$ constant independent of $X$ and $n$.
\end{thm}
\begin{rem}
	The correspondence preceding Theorem \ref{thm-1.2} allows for an extension of Theorem \ref{thm-4.1} to bounded, simply connected, convex complete Reinhardt domains. Furthermore, by combining Theorem \ref{thm-4.1} with Lemma \ref{lem-3.5}, we obtain asymptotic estimates for all bounded, simply connected complete Reinhardt domains, not necessarily convex. 
\end{rem}
\subsection{Proof of lower bounds of Theorem \ref{thm-4.1}}
	We first establish the lower bound. We split the proof into several cases. 
	Let $H$, $G$, and $P$ be the functions defined in the proof of Theorem \ref{thm-1.2}. Let $X=\mathcal{B}(\mathcal{H})$, for short. Let $1 \leq p,q < \infty$.
	
	\medskip
	\noindent
	In view of \eqref{snd-e-4.6}, it is enough to estimate
	\[
	\sum_{|\alpha|=m}\big(\|a_\alpha\|^p+\|b_\alpha\|^p\big).
	\]
	
	\medskip
	\noindent
\underline{\textbf{Case $p >q$:}}
	Let
	\[
	g(z)=\sum_{|\alpha|=m} c_\alpha z^\alpha
	\]
	be a scalar-valued holomorphic polynomial on $B_{\ell_q^n}$ such that
	$\|g\|_{B_{\ell_q^n}}\le 1$. Then, by \cite[(3
	.7), p. 24]{defant-2006}, we have
	\begin{equation}\label{scalar-coeff}
		|c_\alpha|
		\le e^{\frac m q}\Big(\frac{m!}{\alpha!}\Big)^{\frac1q},
		\qquad |\alpha|=m.
	\end{equation}
	\noindent
	Let 
	$Id:\ell_q^n\longrightarrow Z$
	be the identity operator. By definition of the operator norm,
	\[
	\|z\|_Z=\|Iz\|_Z\le \|Id:\ell_q^n\to Z\|\,\|z\|_{\ell_q^n},
	\qquad z\in\mathbb C^n.
	\]
	Hence
	\[
	\frac{1}{\|Id:\ell_q^n\to Z\|}B_{\ell_q^n}\subset B_Z,
	\qquad\text{equivalently}\qquad
	B_{\ell_q^n}\subset \|Id:\ell_q^n\to Z\|\, B_Z.
	\]
	
	Let
	\[
	g_1(z)=\sum_{|\alpha|=m} d_\alpha z^\alpha
	\]
	be an $m$-homogeneous scalar-valued polynomial on $B_Z$ such that
	$\|g\|_{B_Z}\le 1$. Define
	\[
	h(w):=g_1\Big(\frac{w}{\|Id:\ell_q^n\to Z\|}\Big),
	\qquad w\in B_{\ell_q^n}.
	\]
	Since $\frac{w}{\|Id:\ell_q^n\to Z\|}\in B_Z$ for every $w\in B_{\ell_q^n}$,
	we have
	\[
	\norm{h}_{B_{\ell_q^n},\mathbb{C}}\le \norm{g_1}_{B_Z,\mathbb{C}}\le 1.
	\]
	Moreover, using $m$-homogeneity,
	\[
	h(w)=g_1\Big(\frac{w}{\|Id:\ell_q^n\to Z\|}\Big)
	=\sum_{|\alpha|=m} c_\alpha 
	\Big(\frac{w}{\|Id:\ell_q^n\to Z\|}\Big)^\alpha
	=\sum_{|\alpha|=m}\frac{c_\alpha}{\|Id:\ell_q^n\to Z\|^m}w^\alpha.
	\]
	Thus the coefficients of $h$ are $c_\alpha/\|I:\ell_q^n\to Z\|^m$.
	Applying the scalar estimate \eqref{scalar-coeff} on $B_{\ell_q^n}$ gives
	\[
	\frac{|c_\alpha|}{\|Id:\ell_q^n\to Z\|^m}
	\le e^{\frac mq}\Big(\frac{m!}{\alpha!}\Big)^{\frac1q},
	\qquad |\alpha|=m,
	\]
	and hence
	\begin{equation}\label{scalar-Z}
		|c_\alpha|
		\le e^{\frac mq}\Big(\frac{m!}{\alpha!}\Big)^{\frac1q}
		\|Id:\ell_q^n\to Z\|^m.
	\end{equation}
	
	Let
	\[
	H_1(z)=\sum_{|\alpha|=m} x_\alpha z^\alpha,
	\qquad z\in B_Z,
	\]
	be an $X$-valued $m$-homogeneous polynomial.
	For each $\varphi\in X^*$ with $\|\varphi\|\le 1$, the scalar-valued polynomial
	$\varphi\circ H_1$ satisfies
	\[
	\|\varphi\circ H_1\|_{B_Z,\mathbb{C}}\le \|H_1\|_{B_Z,X}.
	\]
	Applying \eqref{scalar-Z} to $\varphi\circ H_1$, we obtain
	\[
	|\varphi(x_\alpha)|
	\le e^{\frac mq}\Big(\frac{m!}{\alpha!}\Big)^{\frac1q}
	\|Id:\ell_q^n\to Z\|^m\,\|H_1\|_{B_Z,X}.
	\]
	Taking the supremum over $\varphi\in X^*$ with $\|\varphi\|\le 1$
	and using the Hahn--Banach theorem yields
	\begin{equation}\label{HB}
		\|x_\alpha\|
		\le e^{\frac mq}\Big(\frac{m!}{\alpha!}\Big)^{\frac1q}
		\|Id:\ell_q^n\to Z\|^m\,\|H_1\|_{B_Z,X},
		\qquad |\alpha|=m.
	\end{equation}
Applying this estimate to the functions $H$ and $G$, where $x_\alpha=(a_\alpha+b_\alpha)/2$ and $(a_\alpha-b_\alpha)/2$, respectively, we have
 \begin{align*}\label{HB-1}
 &	\Big\|\frac{a_\alpha+b_\alpha}{2}\Big\|
 	\le e^{\frac mq}\Big(\frac{m!}{\alpha!}\Big)^{\frac1q}
 	\|Id:\ell_q^n\to Z\|^m\,\|H\|_{B_Z,X},
 	\qquad |\alpha|=m,\\
 & \Big\|\frac{a_\alpha-b_\alpha}{2}\Big\|
 \le e^{\frac mq}\Big(\frac{m!}{\alpha!}\Big)^{\frac1q}
 \|Id:\ell_q^n\to Z\|^m\,\|G\|_{B_Z,X},
 \qquad |\alpha|=m.
 \end{align*}	
	By Lemma \ref{lem-3.1},
	\[
	\|H\|_{B_Z,X}\le 2\|P\|_{B_Z,X}\le 2,
	\]
	and therefore
	\begin{equation}\label{plus-q}
		\sum_{|\alpha|=m}\Big\|\frac{a_\alpha+b_\alpha}{2}\Big\|^q
		\le 2^q e^{m}\,\|Id:\ell_q^n\to Z\|^{mq} \, n^m.
	\end{equation}
Similarly, using the fact $\|G\|_{B_Z,X}\le 2\|P\|_{B_Z,X}\le 2$, we obtain
	\begin{equation}\label{minus-q}
		\sum_{|\alpha|=m}\Big\|\frac{a_\alpha-b_\alpha}{2}\Big\|^q
		\le 2^q e^{m}\,\|Id:\ell_q^n\to Z\|^{mq} \, n^m.
	\end{equation}
	Since $p>q$, the monotonicity of $\ell^r$-norms gives
	\[
	\sum_{|\alpha|=m}\Big\|\frac{a_\alpha\pm b_\alpha}{2}\Big\|^p
	\le \left(\sum_{|\alpha|=m}\Big\|\frac{a_\alpha\pm b_\alpha}{2}\Big\|^q\right)^{\frac pq}.
	\]
	Combining this with \eqref{plus-q} and \eqref{minus-q}, we get
	\begin{align*}
		\sum_{|\alpha|=m}\Big\|\frac{a_\alpha+b_\alpha}{2}\Big\|^p
		&\le 2^p e^{\frac{mp}{q}}\|Id:\ell_q^n\to Z\|^{mp}\, n^{\frac{mp}{q}}, \\
		\sum_{|\alpha|=m}\Big\|\frac{a_\alpha-b_\alpha}{2}\Big\|^p
		&\le 2^p e^{\frac{mp}{q}}\|Id:\ell_q^n\to Z\|^{mp}\, n^{\frac{mp}{q}}.
	\end{align*}
Using the last two inequalities and the elementary fact
\[
\Big\|\sum_{|\alpha|=m}a_\alpha\Big\|_p\le \Big\|\sum_{|\alpha|=m} \frac{a_\alpha+b_\alpha}{2}\Big\|_p+\Big\|\sum_{|\alpha|=m}\frac{a_\alpha-b_\alpha}{2}\Big\|_p,
\]
we deduce that
\begin{align*}
&\sum_{|\alpha|=m}\|a_\alpha\|^p
\le 2^{2p} e^{\frac{mp}{q}}\|Id:\ell_q^n\to Z\|^{mp} n^{\frac{mp}{q}},\\
& \sum_{|\alpha|=m}\|b_\alpha\|^p
\le 2^{2p} e^{\frac{mp}{q}}\|Id:\ell_q^n\to Z\|^{mp} n^{\frac{mp}{q}}.
\end{align*}
Substituting these bounds into \eqref{snd-e-4.6}, for $r \in (0,1)$, we obtain
\begin{align*}
& \sum_{|\alpha|=m} \sup_{z \in r \, B_{Z}} (\norm{a_{\alpha}}^p+\norm{b_{\alpha}}^p)|z|^{p\alpha}\\
& \leq r^{pm}\,2^{2p+1} e^{\frac{mp}{q}}\, \norm{Id: Z \rightarrow \ell^n_q}^{mp}\, \|Id:\ell_q^n\to Z\|^{mp}\, n^{\frac{mp}{q}}. 
\end{align*} 
Hence the right-hand side is bounded by $1$ whenever
\[
r\le \frac{1}{2^{\frac{2p+1}{pm}}e^{\frac{1}{q}}\, \norm{Id: Z \rightarrow \ell^n_q}\,\|Id:\ell_q^n\to Z\|\, n^{\frac{1}{q}}}.
\]
This gives 
\[
L(B_{Z}, p,X)\geq \frac{1}{2^{\frac{2p+1}{pm}}e^{\frac{1}{q}}\, \norm{Id: Z \rightarrow \ell^n_q}\,\|Id:\ell_q^n\to Z\|\, n^{\frac{1}{q}}}.
\]
Finally, in view of Lemma \ref{lem-3.3} (2), we deduce that
\begin{equation} \label{e-4.55}
		L_{\lambda}(B_{Z}, p,X) \geq \dfrac{\left(\lambda^p -1\right)^{\frac{1}{p}}}{\lambda}\, 2^{-\frac{2p+1}{p}}\, e^{-\frac{1}{q}}\,  \frac{1}{ \norm{Id: Z \rightarrow \ell^n_q}\,\|Id:\ell_q^n\to Z\|\, n^{\frac{1}{q}}}.
\end{equation}
This completes the proof for the case $p>q$.

\medskip
\noindent
\underline{\textbf{Case $p > \operatorname{Cot}(X)$:}} We first recall the following fact for any complex Banach space $X$ due to \cite{carando-2020}.
Let 
\[
\widetilde{H}(z)=\sum_{|\alpha|=m} x_{\alpha} z^\alpha
\]
be an $X$-valued $m$-homogeneous polynomial, where $X$ is a complex Banach space. 
Assume that $X$ has finite cotype $t$. Applying \cite[Theorem 3.1]{carando-2020} to $\widetilde{H}$, we obtain
\begin{equation} \label{e-1.35}
	\left(\sum_{|\alpha|=m} \|x_{\alpha}\|^t\right)^{\frac{1}{t}}
	\leq c^m \left(\int_{\mathbb{T}^n} \|\widetilde{H}(z)\|^t \, dz\right)^{\frac{1}{t}}
	\leq c^m \|\widetilde H\|_{\mathbb{D}^n,X},
\end{equation}
for some constant $c>0$ independent of $m$.

Let $p>\mathrm{Cot}(X)$. Then there exists $t\in[\mathrm{Cot}(X),p)$ such that $X$ has cotype $t$, and hence $X$ also has cotype $p$. Consequently, from \eqref{e-1.35} we obtain
\begin{equation}\label{eq:disc-est}
	\sum_{|\alpha|=m}\|x_{\alpha}\|^p
	\leq c^{mp}\|\widetilde H\|_{\mathbb{D}^n,X}^p.
\end{equation}

\medskip

We now transfer this estimate from the polydisc $\mathbb{D}^n$ to the unit ball of an arbitrary finite-dimensional Banach space. Recall that $Z=(\mathbb C^n,\|\cdot\|)$, and its unit ball is $B_Z$.
Consider the identity operator
\[
Id:(\mathbb C^n,\|\cdot\|_\infty)=\ell_\infty^n \longrightarrow Z.
\]
Then
\[
\|z\|\le \|Id:\ell_\infty^n\to Z\|\,\|z\|_\infty
\quad \text{for all } z\in\mathbb C^n,
\]
and therefore
\[
\frac{1}{\|Id:\ell_\infty^n\to Z\|}\mathbb D^n\subset B_Z,
\qquad\text{equivalently}\qquad
\mathbb D^n\subset \|Id:\ell_\infty^n\to Z\|\, B_Z.
\]
Note that, if $z\in \mathbb D^n$, then
$z=\|I:\ell_\infty^n\to Z\|\,w$ for some $w\in B_Z$. By homogeneity,
\[
\widetilde{H}(z)=\widetilde{H}\big(\|I:\ell_\infty^n\to Z\|\,w\big)
=\|I:\ell_\infty^n\to Z\|^m \widetilde{H}(w),
\]
and hence
\[
\|\widetilde{H}(z)\|
\le \|I:\ell_\infty^n\to Z\|^m \sup_{w\in B_Z}\|\widetilde{H}(w)\|.
\]
Taking supremum over $z\in\mathbb D^n$ yields
\[
\|\widetilde{H}\|_{\mathbb D^n,X}
\le \|I:\ell_\infty^n\to Z\|^m \|\widetilde{H}\|_{B_Z,X}.
\]
Combining this with \eqref{eq:disc-est}, we obtain
\begin{equation}\label{eq:ballZ}
	\sum_{|\alpha|=m}\|x_\alpha\|^p
	\le (c\,\|I:\ell_\infty^n\to Z\|)^{mp}\|\widetilde{H}\|_{B_Z,X}^p.
\end{equation}

\medskip

Finally, we apply \eqref{eq:ballZ} to the $m$-homogeneous polynomials
\[
H(z)=\sum_{|\alpha|=m}\frac{a_\alpha+b_\alpha}{2}\,z^\alpha,
\qquad
G(z)=\sum_{|\alpha|=m}\frac{a_\alpha-b_\alpha}{2}\,z^\alpha,
\quad z\in B_Z.
\]
Then
\[
\sum_{|\alpha|=m}\left\|\frac{a_\alpha+b_\alpha}{2}\right\|^p
\le (c_1\,\|I:\ell_\infty^n\to Z\|)^{mp}\|H\|_{B_Z,X}^p,
\]
and
\[
\sum_{|\alpha|=m}\left\|\frac{a_\alpha-b_\alpha}{2}\right\|^p
\le (c_2\,\|I:\ell_\infty^n\to Z\|)^{mp}\|G\|_{B_Z,X}^p
\]
for some constants $c_1,\,c_2>0$, independent of $m$.
Using these and the facts $\|H\|_{B_Z,X},\, \|G\|_{B_Z,X}\le 2\|P\|_{B_Z,X}\le 2$, we deduce that
\begin{align*}
	&\sum_{|\alpha|=m}\|a_\alpha\|^p
	\le 2^{p} \, \|Id:\ell_\infty^n\to Z\|^{mp} \,(c^m_1+c^m_2)^{p},\\
	& \sum_{|\alpha|=m}\|b_\alpha\|^p
	\le 2^{p} \, \|Id:\ell_\infty^n\to Z\|^{mp} \, (c^m_1+c^m_2)^{p}.
\end{align*}
These inequalities, together with \eqref{snd-e-4.6}, for $r \in (0,1)$, yield
\begin{align*}
	& \sum_{|\alpha|=m} \sup_{z \in r \, B_{Z}} (\norm{a_{\alpha}}^p+\norm{b_{\alpha}}^p)|z|^{p\alpha}\\
	& \leq r^{pm}\,2^{2p+1} d^{mp}\, \norm{Id: Z \rightarrow \ell^n_\infty}^{mp}\, \|Id:\ell_\infty^n\to Z\|^{mp}, 
\end{align*} 
where $d= \max\{c_1,c_2\}$ is independent of $n$.
Hence the right-hand side is bounded by $1$ whenever
\[
r\le \frac{1}{2^{\frac{2p+1}{pm}}\,d\, \norm{Id: Z \rightarrow \ell^n_\infty}\,\|Id:\ell_\infty^n\to Z\|},
\]
which shows
\[
L^m(B_{Z}, p,X)\geq \frac{1}{2^{\frac{2p+1}{pm}}\, d\, \norm{Id: Z \rightarrow \ell^n_\infty}\,\|Id:\ell_\infty^n\to Z\|}.
\]
Finally, in view of Lemma \ref{lem-3.3} (2), we conclude that
\begin{equation*}
	L_{\lambda}(B_{Z}, p,X) \geq \dfrac{\left(\lambda^p -1\right)^{\frac{1}{p}}}{\lambda}\, 2^{-\frac{2p+1}{p}}\, d^{-1}\,  \frac{1}{ \norm{Id: Z \rightarrow \ell^n_\infty}\,\|Id:\ell_\infty^n\to Z\|}.
\end{equation*}
This completes the proof for the case $p>\operatorname{Cot}(X)$.

\medskip
\noindent
\underline{\textbf{Case $p \leq \operatorname{Cot}(X)$:}}
We now consider the case $p \leq \operatorname{Cot}(X)$. Then $p \leq t$ whenever $X$ has cotype $t$. 
Let $P$ be as defined in Theorem \ref{thm-1.2}. Proceeding in a manner similar to the discussion in the case $p>\operatorname{Cot}(X)$, we obtain
\begin{equation}\label{eq:main1}
	\sum_{|\alpha|=m}\sup_{z\in B_Z} \left( \|a_\alpha\|^t+\|b_\alpha\|^t \right)|\delta z|^{t\alpha}\leq 1,
\end{equation}
for
\[
\delta=\frac{c}{\|Id: Z\to \ell_\infty^n\|\,\|Id:\ell_\infty^n\to Z\|},
\]
where $c>0$ is an absolute constant.

Let $r=\gamma \, \delta \in(0,1)$. Using H\"{o}lder's inequality together with \eqref{eq:main1}, we obtain
\begin{align*}
	\sum_{|\alpha|=m}\sup_{z\in r \, B_Z}\left(\|a_\alpha\|^p+\|b_\alpha\|^p\right)|z|^{p\alpha}
	&=\sum_{|\alpha|=m}\sup_{z\in B_Z}\left(\|a_\alpha\|^p+\|b_\alpha\|^p\right)|\gamma \, \delta z|^{p\alpha}\\
	&= \gamma^{pm}\sum_{|\alpha|=m}\sup_{z\in B_Z}\left(\|a_\alpha\|^p+\|b_\alpha\|^p\right)|\delta z|^{p\alpha}\\
	&\leq \gamma^{pm}\left[\sum_{|\alpha|=m}\sup_{z\in B_Z}\|a_\alpha\|^t|\delta z|^{t\alpha}\right]^{\frac{p}{t}}
	\left(\sum_{|\alpha|=m}1\right)^{1-\frac{p}{t}}\\
	&\quad + \gamma^{pm}\left[\sum_{|\alpha|=m}\sup_{z\in B_Z}\|b_\alpha\|^t|\delta z|^{t\alpha}\right]^{\frac{p}{t}}
	\left(\sum_{|\alpha|=m}1\right)^{1-\frac{p}{t}}.
\end{align*}
By \eqref{eq:main1}, both sums inside the brackets are bounded by $1$, and hence
\begin{equation}\label{eq:main2}
	\sum_{|\alpha|=m}\sup_{z\in B_Z}\left(\|a_\alpha\|^p+\|b_\alpha\|^p\right)|rz|^{p\alpha}
	\leq 2\, \gamma^{pm}\left(\sum_{|\alpha|=m}1\right)^{1-\frac{p}{t}}.
\end{equation}
Using the following standard estimate
\[
\sum_{|\alpha|=m}1=\binom{n+m-1}{m}\leq e^m n^m,
\]
we deduce from \eqref{eq:main2} that
\[
\sum_{|\alpha|=m}\sup_{z\in B_Z}\left(\|a_\alpha\|^p+\|b_\alpha\|^p\right)|rz|^{p\alpha}
\leq C_3^m \gamma^{pm} n^{m\left(1-\frac{p}{t}\right)},
\]
for some constant $C_3>0$ independent of $m$.

The right-hand side is less than or equals to $1$ provided
\[
\gamma\leq C_4 \, n^{\frac{1}{t}-\frac{1}{p}},
\]
for some constant $C_4>0$. Consequently,
\[
r\leq C_5 \, \frac{ n^{\frac{1}{t}-\frac{1}{p}}}{\|Id: Z\to \ell_\infty^n\|\,\|Id:\ell_\infty^n\to Z\|},
\]
for some constant $C_5>0$ independent of $n$. This shows
\[
L^m(B_{Z}, p,X)\geq C_5 \, \frac{ n^{\frac{1}{t}-\frac{1}{p}}}{\|Id: Z\to \ell_\infty^n\|\,\|Id:\ell_\infty^n\to Z\|}.
\]
Finally, in view of Lemma \ref{lem-3.3} (2), we conclude that
\begin{equation*}
	L_{\lambda}(B_{Z}, p,X) \geq C_5 \, \dfrac{\left(\lambda^p -1\right)^{\frac{1}{p}}}{\lambda}\,  \, \frac{ n^{\frac{1}{t}-\frac{1}{p}}}{\|Id: Z\to \ell_\infty^n\|\,\|Id:\ell_\infty^n\to Z\|}.
\end{equation*}
This completes the proof for the case $p\leq \operatorname{Cot}(X)$.

\medskip

\noindent
\underline{\textbf{Case $p\leq q$:}}
Proceeding similarly as in the case $p>q$, we obtain
\[
\sum_{|\alpha|=m}\sup_{z\in B_Z}\left(\|a_\alpha\|^q+\|b_\alpha\|^q\right)|\delta z|^{q\alpha}\leq 1,
\]
for
\[
\delta=\frac{C_6}{\|I_d: Z\to \ell_q^n\|\,\|I_d:\ell_q^n\to Z\|\,n^{1/q}},
\]
where $C_6>0$ is a constant. Repeating the above argument, we obtain
\[
L^m(B_Z,p,X)\geq C_7\,
\frac{1}{\|Id: Z\to \ell_q^n\|\,\|Id:\ell_q^n\to Z\|\,n^{1/p}},
\]
for some constant $C_7>0$ independent of $n$. 
Finally, in view of Lemma \ref{lem-3.3} (2), we conclude that
\begin{equation*}
	L_{\lambda}(B_{Z}, p,X) \geq C_7\, \dfrac{\left(\lambda^p -1\right)^{\frac{1}{p}}}{\lambda}\,   \frac{1}{ \norm{Id: Z \rightarrow \ell^n_q}\,\|Id:\ell_q^n\to Z\|\, \, n^{1/p}}.
\end{equation*}
This completes the proof.

\medskip
In order to prove the upper bound, we shall make use of the following result.
\begin{customthm}{A} \cite[Theorem 14.5]{diestel-abs-summing-1995} \label{thm-A}
	Given any infinite dimensional complex Banach space $X$, there exist $x_{1}, \ldots , x_{n} \in X$ for each $n \in \mathbb{N}$ such that $\norm{z}_{\infty}/2 \leq \norm{\sum_{j=1}^{n} x_{j}z_{j}} \leq \norm{z}_{Cot(X)}$  for every choice of $z=(z_{1}, \ldots,z_{n}) \in \mathbb{C}^n$. Clearly, setting $z=e_{j}$, gives $\norm{x_{j}}\geq 1/2$, where $e_{j}$ is the $j$-th canonical basis vector of $\mathbb{C}^n$.
\end{customthm}
\subsection{Proof of upper bounds of Theorem \ref{thm-4.1}}
We now derive an upper bound for $L_{\lambda}(B_{Z}, p,X)$. Since $L_{\lambda}(\mathbb{D}^n, p,X)=R_{\lambda}(\mathbb{D}^n, p,X)$, it suffices to estimate $R_{\lambda}(\mathbb{D}^n, p,X)$. The desired upper bound for $L_{\lambda}(B_Z, p,X)$ will then follow from Lemma \ref{lem-3.5}. 

\medskip
By evaluating Theorem \ref{thm-A} at $z=e_j$, we obtain that for each $n \in \mathbb{N}$ there exist vectors $a_{1}, \ldots, a_{n} \in X$ such that \[ \|a_{j}\| \geq \frac{1}{2}, \qquad 1\leq j \leq n. \] Consider the $1$-homogeneous pluriharmonic polynomial $F \in \mathcal{PH}(^1\mathbb{D}^n,X)$ defined by 
\[ F(z)=\sum_{j=1}^{n} z_{j}a_{j} + \sum_{j=1}^{n} a^{*}_{j}\,\overline{z_{j}}. 
\] From the definition of $R^{1}(\mathbb{D}^n,p,X)$ and a direct computation, we obtain \begin{equation}\label{e-1.39} \frac{2n}{2^{p}} \leq \sum_{j=1}^{n} \bigl(\|a_{j}\|^{p}+\|a_{j}\|^{p}\bigr) \leq \frac{\|F\|^{p}_{\mathbb{D}^n,X}} {\bigl(R^{1}(\mathbb{D}^n,p,X)\bigr)^{p}}. 
\end{equation} 
Applying Theorem \ref{thm-A} once more, we have \[ \|F(z)\| \leq 2 \left\| \sum_{j=1}^{n} z_{j}a_{j} \right\| \leq 2\,\|z\|_{\mathrm{Cot}(X)}. \]
 Combining this estimate with \eqref{e-1.39}, we deduce that \[ \frac{2n}{2^{p}} \leq \frac{2^{p}} {\bigl(R^{1}(\mathbb{D}^n,p,X)\bigr)^{p}} \sup_{z\in \mathbb{D}^n} \|z\|_{\mathrm{Cot}(X)}^{p}. \] Since \[ \sup_{z\in \mathbb{D}^n}\|z\|_{\ell} = n^{1/\ell}, \qquad 1\leq \ell \leq \infty, \] it follows from the previous inequality that \[ R^{1}(\mathbb{D}^n,p,X) \leq 2^{\frac{2p-1}{p}}\, \frac{1}{n^{\frac{1}{p}-\frac{1}{\mathrm{Cot}(X)}}}. \] It is known that \[ R_{\lambda}(\mathbb{D}^n,p,X) \leq R^{1}_{\lambda}(\mathbb{D}^n,p,X) =\lambda\,R^{1}(\mathbb{D}^n,p,X), \] and therefore \[ L_{\lambda}(\mathbb{D}^n,p,X) = R_{\lambda}(\mathbb{D}^n,p,X) \leq 2^{\frac{2p-1}{p}}\, \lambda\, \frac{1}{n^{\frac{1}{p}-\frac{1}{\mathrm{Cot}(X)}}}. \] Finally, applying Lemma \ref{lem-3.5}, we conclude that \[ L_{\lambda}(B_Z,p,X) \leq C\,\lambda\, \frac{1}{n^{\frac{1}{p}-\frac{1}{\mathrm{Cot}(X)}}} \,\| \mathrm{Id}: Z \to \ell_\infty^n \|\, \|\mathrm{Id}: \ell_\infty^n \to Z\| \] 
 for some constant $C>0$ independent both of $X$ and $n$. This completes the proof.
\section{Applications to Banach sequence spaces}
This section presents several applications of our results to various classes of Banach sequence spaces, which have wide applications.

\medskip
A Banach space $Z$ satisfying $\ell_1 \subset Z \subset c_{0}$ with normal inclusions is called a Banach sequence space if its canonical basis vectors $\{e_{k}\}$ forms a $1$-unconditional basis. A Banach lattice $Z$ is said to be $2$-convex if there exists a constant $C>0$ such that 
$$\norm{\left(\sum_{k=1}^{n}|x_{k}|^2\right)^{\frac{1}{2}}} \leq C 	\left(\sum_{k=1}^{n}\norm{x_{k}}^2\right)^{\frac{1}{2}}
$$
for all finite families $x_{1}, \ldots, x_{k} \in Z$. 

\medspace
\noindent As an application of Theorem \ref{thm-1.2}, we deduce the following asymptotic estimates for Banach sequence spaces in finite dimensional settings.
\begin{thm} \label{thm-1.3-a}
	Let $Z$ be a Banach sequence space, $X=\mathcal{B}(\mathcal{H})$ be finite dimensional, and $\lambda>1$. For $n\in \mathbb{N}$, let $Z_{n}$ be the linear span of $\{e_{k}, \, k=1,\ldots,n\}$. Then for $1\leq q \leq \infty$
	\begin{enumerate}
		\item if $Z$ is a subset of $\ell_2$ we have \begin{equation*}
			L_{\lambda}(B_{Z_{n}}, p,X) \geq	\begin{cases}
				E_{1}(X)\, \frac{1}{s\,t\, n^{\frac{1}{q}}} \, \max \left\{\frac{1}{\sqrt{n}\, \norm{\mathrm{Id}: \ell^n_{2} \rightarrow Z_{n}}}, \, \frac{1}{e\, \norm{\sum_{k=1}^{n}e^{*}_{k}}_{Z^{*}}}\right\}\, \left(\frac{\lambda^p -1}{2\lambda^p - 1}\right)^{\frac{1}{p}} & \text{for $p=1$}, \\[2mm]
				E_{3}\, \frac{1}{s}\,\norm{\mathrm{Id}:\ell^n _2 \rightarrow Z_{n}}^{-\frac{2}{p}}\, \left(\frac{\lambda^p -1}{2\lambda^p - 1}\right)^{\frac{1}{p}}\, & \text{for $p \geq 2$}, \\[2mm]
				
				\frac{E_{2}(X)}{s\left(t n^{\frac{1}{q}}\right)^{\frac{(2-p)}{p}}}\,  \max \left\{\frac{1}{(\sqrt{n})^{1-\theta} \norm{\mathrm{Id}: \ell^n_{2} \rightarrow Z_{n}}}, \frac{\norm{\mathrm{Id}:\ell^n _2 \rightarrow Z_{n}}^{-\theta}}{e^{(1-\theta)} \norm{\sum_{k=1}^{n}e^{*}_{k}}_{Z^{*}}^{1-\theta}}\right\}\left(\frac{\lambda^p -1}{2\lambda^p - 1}\right)^{\frac{1}{p}} \hspace{-3mm} & \text{for $1<p< 2$}. 
			\end{cases}
		\end{equation*}
		and 
		$$L_{\lambda}(B_{Z_{n}}, p,X) \leq d\,  \lambda^{\frac{1}{\log \,n}}\,b_{[\log\,n]}(B_Z)\, \,n^{-\frac{1}{p}};$$
		\item if $Z$ is symmetric and $2$-convex, we have
		\begin{equation*}
			L_{\lambda}(B_{Z_{n}}, p,X) \geq	\begin{cases}
				E_{1}(X)\, \frac{1}{s\,t\, n^{\frac{1}{q}}} \, \max \left\{\frac{1}{\sqrt{n}}, \, \frac{1}{e\, \norm{\sum_{k=1}^{n}e^{*}_{k}}_{Z^{*}}}\right\}\, \left(\frac{\lambda^p -1}{2\lambda^p - 1}\right)^{\frac{1}{p}} & \text{for $p=1$}, \\[2mm]
				E_{3}\,\frac{1}{s}\, \left(\frac{\lambda^p -1}{2\lambda^p - 1}\right)^{\frac{1}{p}}\, & \text{for $p \geq 2$}, \\[2mm]
				
				\frac{E_{2}(X)}{s\left(t n^{\frac{1}{q}}\right)^{\frac{(2-p)}{p}}}\,  \max \left\{\frac{1}{(\sqrt{n})^{1-\theta}}, \, \frac{1}{e^{(1-\theta)}\, \norm{\sum_{k=1}^{n}e^{*}_{k}}_{Z^{*}}^{1-\theta}}\right\}\left(\frac{\lambda^p -1}{2\lambda^p - 1}\right)^{\frac{1}{p}}  & \text{for $1<p< 2$}. 
			\end{cases}
		\end{equation*}
		and 
		$$L_{\lambda}(B_{Z}, p,X) \leq d\,  \lambda^{\frac{1}{\log \,n}}\,b_{[\log\,n]}(B_Z)\, \norm{\sum_{k=1}^{n}e^{*}_{k}}_{Z^{*}}\,\,n^{-\left(\frac{1}{p}+\frac{1}{2}\right)}.
		$$
	\end{enumerate}
	Here $s$, $t$, and the constants $E_{1}, E_{2}, E_{3}$, and $d$ are same as in Theorem \ref{thm-1.2}.
\end{thm}
\begin{proof} [{\bf Proof of Theorem \ref{thm-1.3-a}}]
	We start by recalling two important facts from the theory of finite-dimensional Banach spaces with unconditional bases. Let $W=(\mathbb{C}^n, ||.||)$ be a Banach space whose canonical unit vectors $e_{k}$ form a normalized $1$-unconditional basis. Then  $\norm{\mathrm{Id}: W \rightarrow \ell_2^n}=\sup_{z \in B_{W}}(\sum_{m=1}^{n}|z_{m}|^2)^{1/2}$. Moreover, $\norm{\mathrm{Id}: W \rightarrow \ell_1^n}=\sup_{z \in B_{W}} \sum_{m=1}^{n}|z_{m}|=\norm{\sum_{m=1}^{n}e^{*}_{m}}_{W^{*}}$ for all $z \in W$. 
	\medspace
	
	Since $Z \subset \ell_2$, we may normalize the inclusion map $Z \subset \ell_2$ to have norm one. Under this assumption, assertion $(1)$ follows immediately from Theorem \ref{thm-1.2}, in combination with the above expressions for the norms of the identity embeddings.
	\vspace{1mm}
	
	We now turn to the proof of part $(2)$. By \cite[Theorem 1.d.5]{Lindenstrauss-book}, we may assume without loss of generality that $M^{(2)}(X)=1$, which in turn implies $\norm{\mathrm{Id}: \ell_2^n \rightarrow Z_{n}}=1$. We next recall an important estimate from  \cite[(5.3), p. 191]{defant-2003}: if $W=(\mathbb{C}^n, ||.||)$ is a symmetric Banach space satisfying $M^{(2)}(W)=1$, then 
	\begin{equation} \label{e-5.4-added}
		\norm{\mathrm{Id}: W \rightarrow \ell^n_2}= \frac{\norm{\sum_{m=1}^{n}e^{*}_{m}}_{W^{*}}}{\sqrt{n}}.
	\end{equation}
	The desired conclusion now follows by combining Theorem \ref{thm-1.2}, the above identities for identity embeddings, and inequality \eqref{e-5.4-added}.The proof is therefore complete.	
\end{proof}
In light of Theorem \ref{thm-4.1} and the arguments employed in the proof of Theorem \ref{thm-1.3-a}, we deduce the following asymptotic estimates for Banach sequence spaces in infinite-dimensional framework. 
\begin{thm} \label{thm-1.6}
	Let $Z$ be a Banach sequence space and $X=\mathcal{B}(\mathcal{H})$ be infinite dimensional. For $n\in \mathbb{N}$, let $Z_{n}$ be the linear span of $\{e_{k}, \, k=1,\ldots,n\}$. Then for $\lambda>1$
	\begin{enumerate}
		\item if $Z$ is subset of $\ell_2$, we have 
		$$  L_{\lambda}(B_{Z_{n}}, p,X) \geq \begin{cases}
			E_{5}\, \dfrac{\left(\lambda^p -1\right)^{\frac{1}{p}}}{\lambda}\, \,  \frac{1}{ \|\mathrm{Id}:\ell_2^n\to Z_n\|\, n^{\frac{1}{2}}} & \text{for $p > 2$}, \\[2mm]
			
			E_{6}\,\dfrac{\left(\lambda^p -1\right)^{\frac{1}{p}}}{\lambda}\, \,  \frac{1}{ \norm{\mathrm{Id}: Z_n \rightarrow \ell^n_\infty}\,\|\mathrm{Id}:\ell_\infty^n\to Z_n\|} & \text{for $p>\mathrm{Cot}(X)$}, \\[2mm]
			E_{7}\, \dfrac{\left(\lambda^p -1\right)^{\frac{1}{p}}}{\lambda}\,  \, \frac{ n^{\frac{1}{t}-\frac{1}{p}}}{\|\mathrm{Id}: Z_n\to \ell_\infty^n\|\,\|\mathrm{Id}:\ell_\infty^n\to Z_n\|} & \text{for $p\leq  \mathrm{Cot}(X)$},\\[2mm]
			E_{8}\, \dfrac{\left(\lambda^p -1\right)^{\frac{1}{p}}}{\lambda}\,   \frac{1}{ \|\mathrm{Id}:\ell_2^n\to Z_n\|\, \, n^{\frac{1}{p}}} & \text{for $p \leq 2$};
		\end{cases}$$
		and
		\begin{equation*}
			L_{\lambda}(B_{Z_n}, p,X) \leq 	
			E_9\,\lambda\, n^{\frac{1}{\mathrm{Cot}(X)}\, - \,\frac{1}{p}}\,\, \| \mathrm{Id}: Z_n \to \ell_\infty^n \|\, \|\mathrm{Id}: \ell_\infty^n \to Z_n\|. 
		\end{equation*}
		\item if $Z$ is symmetric and $2$-convex, we have $$  L_{\lambda}(B_{Z_{n}}, p,X) \geq \begin{cases}
			E_{5}\, \dfrac{\left(\lambda^p -1\right)^{\frac{1}{p}}}{\lambda}\, \,  \frac{1}{ \norm{\sum_{k=1}^{n}e^{*}_{k}}_{Z^{*}}} & \text{for $p > 2$}, \\[2mm]
			
			E_{6}\,\dfrac{\left(\lambda^p -1\right)^{\frac{1}{p}}}{\lambda}\, \,  \frac{1}{ \norm{\mathrm{Id}: Z_n \rightarrow \ell^n_\infty}\,\|\mathrm{Id}:\ell_\infty^n\to Z_n\|} & \text{for $p>\mathrm{Cot}(X)$}, \\[2mm]
			E_{7}\, \dfrac{\left(\lambda^p -1\right)^{\frac{1}{p}}}{\lambda}\,  \, \frac{ n^{\frac{1}{t}-\frac{1}{p}}}{\|\mathrm{Id}: Z_n\to \ell_\infty^n\|\,\|\mathrm{Id}:\ell_\infty^n\to Z_n\|} & \text{for $p\leq  \mathrm{Cot}(X)$},\\[2mm]
			E_{8}\, \dfrac{\left(\lambda^p -1\right)^{\frac{1}{p}}}{\lambda}\,   \frac{1}{ \norm{\sum_{k=1}^{n}e^{*}_{k}}_{Z^{*}}\, \, n^{\frac{1}{p}-\frac{1}{2}}} & \text{for $p \leq 2$};
		\end{cases}$$
		and
		\begin{equation*}
			L_{\lambda}(B_{Z_n}, p,X) \leq 	
			E_9\,\lambda\, n^{\frac{1}{\mathrm{Cot}(X)}\, - \,\frac{1}{p}}\,\, \| \mathrm{Id}: Z_n \to \ell_\infty^n \|\, \|\mathrm{Id}: \ell_\infty^n \to Z_n\|. 
		\end{equation*}	
	\end{enumerate}
	Here the constants $E_5, \ldots, E_9$ are as in Theorem \ref{thm-4.1}.
\end{thm}
We now turn to several standard examples of Banach sequence spaces and apply our results to these settings.

\medskip
 We begin with the family of mixed Minkowski spaces. Recall that
$$\ell^m_s(\ell^n_t):=\{(z_{k})^m_{k=1}:z_{1}, \ldots,z_{m} \in \mathbb{C}^n\}$$ 
equipped with the norm $\norm{(z_{k})^m_{k=1}}_{s,t}:=(\sum_{k=1}^{m}\norm{z_{k}}^s_t)^{1/s}$. In the special case when $s=t$, the spaces $\ell^m_s(\ell^n_t)$ coincide with the classical Minkowski spaces $\ell^{mn}_s$. The framework developed in this article naturally includes this setting as a particular case while simultaneously accommodating a significantly broader class of sequence spaces. This unified perspective allows us to investigate both holomorphic and pluriharmonic functions in general sequence space settings.
\begin{example}[{\bf Mixed spaces}]
	Let $m,n \geq 2$. Then
	\begin{enumerate}
		\item Let $1 \leq s ,t \leq 2$. If $X=\mathcal{B}(\mathcal{H})$ is finite dimensional, then
		\begin{equation*}
			L_{\lambda}(B_{\ell^m_s(\ell^n_t)}, p,X) \geq	\begin{cases}
				E_{1}(X)\,  \, \max \left\{\frac{1}{n^{2/t}\,m^{2/s}}, \, \frac{1}{e\,mn }\right\}\, \left(\frac{\lambda^p -1}{2\lambda^p - 1}\right)^{\frac{1}{p}} & \text{for $p=1$}, \\[2mm]
				E_{3}\, n^{\frac{1}{p}-\frac{2}{tp}}\,m^{\frac{1}{p}-\frac{2}{sp}} \left(\frac{\lambda^p -1}{2\lambda^p - 1}\right)^{\frac{1}{p}}\, & \text{for $p \geq 2$}, \\[2mm]
				
				E_{2}(X)  \max \left\{\frac{1}{ n^{\frac{2+t-tp}{tp}}m^{\frac{2+s-sp}{sp}}},  \frac{m^{\frac{2+2ps-3s-2p}{ps}}}{e^{(1-\theta)} n^{\frac{3t+2p-2-2pt}{pt}}}\right\}\left(\frac{\lambda^p -1}{2\lambda^p - 1}\right)^{\frac{1}{p}} \hspace{-1mm} & \text{for $1<p< 2$}, 
			\end{cases}
		\end{equation*}
		and 
		$$L_{\lambda}(B_{\ell^m_s(\ell^n_t)}, p,X) \leq d\,\lambda^{\frac{1}{\log \,mn}}\,b_{[\log\,mn]}(B_{\ell^m_s(\ell^n_t)})\, \,\frac{1}{(mn)^{\frac{1}{p}}}.
		$$
		If $X=\mathcal{B}(\mathcal{H})$ is infinite dimensional then
		$$L_{\lambda}(B_{\ell^m_s(\ell^n_t)}, p,X) \geq	\begin{cases}
			E_{5}\, \dfrac{\left(\lambda^p -1\right)^{\frac{1}{p}}}{\lambda}\, \,  \frac{1}{ n^{\frac{1}{t}}\, m^{\frac{1}{s}}} & \text{for $p > 2$}, \\[2mm]
			
			E_{6}\,\dfrac{\left(\lambda^p -1\right)^{\frac{1}{p}}}{\lambda}\, \,  \frac{1}{ n^{\frac{1}{t}}\, m^{\frac{1}{s}}} & \text{for $p>\mathrm{Cot}(X)$}, \\[2mm]
			E_{7}\, \dfrac{\left(\lambda^p -1\right)^{\frac{1}{p}}}{\lambda}\,  \, \frac{ 1}{n^{\frac{1}{t}+\frac{1}{p}-\frac{1}{r}}\, m^{\frac{1}{s}+\frac{1}{p}-\frac{1}{r}}} & \text{for $p\leq  \mathrm{Cot}(X)$},\\[2mm]
			E_{8}\, \dfrac{\left(\lambda^p -1\right)^{\frac{1}{p}}}{\lambda}\,   \frac{1}{ n^{\frac{1}{t}+\frac{1}{p}-\frac{1}{2}}\, m^{\frac{1}{s}+\frac{1}{p}-\frac{1}{2}}} & \text{for $p \leq 2$},
		\end{cases}
		$$ 
		whenever $X$ has cotype $r$. Moreover,
		\begin{equation*}
			L_{\lambda}(B_{\ell^m_s(\ell^n_t)}, p,X) \leq E_9\,\lambda\, n^{\frac{1}{\mathrm{Cot}(X)}\, - \,\frac{1}{p}+\frac{1}{t}}\,\, m^{\frac{1}{\mathrm{Cot}(X)}\, - \,\frac{1}{p}+\frac{1}{s}}.
		\end{equation*}
		\item Let $2 \leq s,t \leq \infty$. If $X=\mathcal{B}(\mathcal{H})$ is finite dimensional then
		\begin{equation*}
			R_{\lambda}(B_{\ell^m_s(\ell^n_t)}, p,X) \geq	\begin{cases}
				E_{1}(X)\,  \, \max \left\{\frac{1}{n^{\frac{3}{2}-\frac{1}{t}}m^{\frac{3}{2}-\frac{1}{s}}}, \, \frac{1}{e\, n^{2-\frac{2}{t}}m^{2-\frac{2}{s}}}\right\}\, \left(\frac{\lambda^p -1}{2\lambda^p - 1}\right)^{\frac{1}{p}} & \text{for $p=1$}, \\[2mm]
				E_{3}\,\frac{1}{n^{\frac{1}{2}-\frac{1}{t}}\, m^{\frac{1}{2}-\frac{1}{s}}} \left(\frac{\lambda^p -1}{2\lambda^p - 1}\right)^{\frac{1}{p}}\, & \text{for $p \geq 2$}, \\[2mm]
				
				E_{2}(X)  \max \left\{\frac{m^{\frac{1}{s}-\frac{4-p}{2p}}}{n^{\frac{4-p}{2p}-\frac{1}{t}}}, \frac{m^{\frac{2}{ps}-\frac{3-p}{p}}}{e^{(1-\theta)} n^{\frac{3-p}{p} -\frac{2}{pt}}}\right\}\left(\frac{\lambda^p -1}{2\lambda^p - 1}\right)^{\frac{1}{p}}  \hspace{-2mm} & \text{for $1<p< 2$}, 
			\end{cases}
		\end{equation*}
		and
		$$L_{\lambda}(B_{\ell^m_s(\ell^n_t)}, p,X) \leq d\,\lambda^{\frac{1}{\log \,mn}}\,b_{[\log\,mn]}(B_{\ell^m_s(\ell^n_t)})\,n^{\frac{1}{2}-\frac{1}{t}-\frac{1}{p}}m^{\frac{1}{2}-\frac{1}{s}-\frac{1}{p}}.
		$$
		If $X=\mathcal{B}(\mathcal{H})$ is infinite dimensional then
		$$L_{\lambda}(B_{\ell^m_s(\ell^n_t)}, p,X) \geq	\begin{cases}
			E_{5}\, \dfrac{\left(\lambda^p -1\right)^{\frac{1}{p}}}{\lambda}\, \,  \frac{1}{ n^{1-\frac{1}{t}}\, m^{1-\frac{1}{s}}} & \text{for $p > 2$}, \\[2mm]
			
			E_{6}\,\dfrac{\left(\lambda^p -1\right)^{\frac{1}{p}}}{\lambda}\, \,  \frac{1}{ n^{\frac{1}{t}}\, m^{\frac{1}{s}}} & \text{for $p>\mathrm{Cot}(X)$}, \\[2mm]
			E_{7}\, \dfrac{\left(\lambda^p -1\right)^{\frac{1}{p}}}{\lambda}\,  \, \frac{ 1}{n^{\frac{1}{t}+\frac{1}{p}-\frac{1}{r}}\, m^{\frac{1}{s}+\frac{1}{p}-\frac{1}{r}}} & \text{for $p\leq  \mathrm{Cot}(X)$},\\[2mm]
			E_{8}\, \dfrac{\left(\lambda^p -1\right)^{\frac{1}{p}}}{\lambda}\,   \frac{1}{ n^{\frac{1}{2}-\frac{1}{t}+\frac{1}{p}}\, m^{\frac{1}{2}-\frac{1}{s}+\frac{1}{p}}} & \text{for $p \leq 2$},
		\end{cases}
		$$ 
		whenever $X$ has cotype $r$. Moreover,
		\begin{equation*}
			L_{\lambda}(B_{\ell^m_s(\ell^n_t)}, p,X) \leq E_9\,\lambda\, n^{\frac{1}{\mathrm{Cot}(X)}\, - \,\frac{1}{p}+\frac{1}{t}}\,\, m^{\frac{1}{\mathrm{Cot}(X)}\, - \,\frac{1}{p}+\frac{1}{s}}.
		\end{equation*}
	\end{enumerate}
	Here the constants $E_1, \ldots, E_9$, and $d$ are as in Theorems \ref{thm-1.2} and \ref{thm-4.1}.	
\end{example}

\begin{pf}
	We begin by observing the following fact that for $1 \le s,t,\ell,r \le \infty$.
	A direct computation (see \cite[ p.~188]{defant-2003}) shows that the identity operator 
	between mixed spaces factorizes as
	\[
	\| \mathrm{Id} : \ell_s^m(\ell_t^n) \to \ell_\ell^m(\ell_r^n) \|
	=
	\| \mathrm{Id} : \ell_s^m \to \ell_\ell^m \|
	\,
	\| \mathrm{Id} : \ell_t^n \to \ell_r^n \|.
	\]
	This follows from the product structure of the norm:
	\[
	\|(x_{ij})\|_{\ell_s^m(\ell_t^n)}
	=
	\left( \sum_{i=1}^m 
	\left( \sum_{j=1}^n |x_{ij}|^t \right)^{s/t}
	\right)^{1/s},
	\]
	which allows the identity operator to decompose coordinatewise.
	
	Moreover, for finite-dimensional $\ell_p$-spaces it is well known that
	\[
	\| \mathrm{Id} : \ell_s^n \to \ell_t^n \|
	=
	\begin{cases}
		1, & s \le t, \\[6pt]
		n^{\frac{1}{t}-\frac{1}{s}}, & s > t.
	\end{cases}
	\]
	Consequently,
	\[
	\| \mathrm{Id} : \ell_s^m(\ell_t^n) \to \ell_\ell^m(\ell_r^n) \|
	=
	m^{\max\{ \frac{1}{\ell}-\frac{1}{s},\,0\}}
	\,
	n^{\max\{ \frac{1}{r}-\frac{1}{t},\,0\}}.
	\]
	
	With this explicit norm computation at hand, the lower and upper estimates 
	for $L_\lambda\big(B_{\ell_s^m(\ell_t^n)}, p, X\big)$	
	follow directly from Theorems \ref{thm-1.2} and \ref{thm-4.1} by choosing $q=2$ and transferring the bounds via the above identity 
	operator estimates.
	\par
	More precisely, the exponents of $m$ and $n$ appearing in each case 
	are obtained by inserting the norm 
	\[
	\| \mathrm{Id} : \ell_s^m(\ell_t^n) \to \ell_2^{mn} \|
	\]
	into the general bounds provided in 
	\ref{thm-1.2} and \ref{thm-4.1}. 
	\par
	In the infinite-dimensional case $X = \mathcal{B}(\mathcal{H})$, the argument is identical, 
	except that we combine the above identity estimates with the cotype 
	assumptions on $X$ and the corresponding parts of Theorems \ref{thm-1.2} and \ref{thm-4.1}. 
	The parameters involving $\mathrm{Cot}(X)$ arise precisely from those 
	general results when applied to spaces with finite cotype $r$.
	\par
	Therefore, all the stated lower and upper bounds follow immediately 
	from Theorems \ref{thm-1.2} and \ref{thm-4.1} together with the explicit computation of the identity operator norm between mixed $\ell_p$-spaces. This completes the proof.
\end{pf}

\medskip
Our results can also be applied to many concrete symmetric Banach sequence spaces such as Lorentz spaces $\ell_{s,t}$ as well as Orlicz spaces $\ell_{\psi}$. For their definition we refer \cite{Lindenstrauss-book}. 
\begin{example}[{\bf Lorentz spaces}]
	Let $1 \leq s,t \leq \infty$.
	\begin{enumerate}
		\item Let $1 \leq s <2$ and $1\leq t \leq \infty$, or $s=2$ and $1 \leq t \leq 2$. If $X=\mathcal{B}(\mathcal{H})$ is finite dimensional 
		\begin{equation*}
			L_{\lambda}(B_{\ell^n_{s,t}}, p,X) \geq	\begin{cases}
				E_{1}(X)\,  \, \max \left\{\frac{1}{n^\frac{2}{s}}, \, \frac{1}{e\, n}\right\}\, \left(\frac{\lambda^p -1}{2\lambda^p - 1}\right)^{\frac{1}{p}} & \text{for $p=1$}, \\[2mm]
				E_{3}\, n^{-\frac{2}{p} \left(\frac{1}{s}-\frac{1}{2}\right)}\, \left(\frac{\lambda^p -1}{2\lambda^p - 1}\right)^{\frac{1}{p}}\, & \text{for $p \geq 2$}, \\[2mm]
				
				E_{2}(X)  \max \left\{\frac{1}{n^{\frac{2-p}{p}} n^{\frac{2}{p}\left(\frac{1}{s}-\frac{1}{2}\right)}},  \frac{n^{-\left(\frac{1}{s}-\frac{1}{2}\right)}}{e^{(1-\theta)} n^{(1-\theta)\left(\frac{3}{2}-\frac{1}{s}\right)}}\right\}\left(\frac{\lambda^p -1}{2\lambda^p - 1}\right)^{\frac{1}{p}} \hspace{-2mm} & \text{for $1<p< 2$}, 
			\end{cases}
		\end{equation*}
		and 
		$$L_{\lambda}(B_{\ell^n_{s,t}}, p,X) \leq d\,\lambda^{\frac{1}{\log \,n}}\,b_{[\log\,n]}(B_{\ell^n_{s,t}})\, \,\frac{1}{n^{\frac{1}{p}}}.
		$$
		If $X=\mathcal{B}(\mathcal{H})$ is infinite dimensional, then
		$$  L_{\lambda}(B_{\ell^n_{s,t}}, p,X) \geq \begin{cases}
			E_{5}\, \dfrac{\left(\lambda^p -1\right)^{\frac{1}{p}}}{\lambda}\, \,  \frac{1}{  n^{\frac{1}{s}}} & \text{for $p > 2$}, \\[2mm]
			
			E_{6}\,\dfrac{\left(\lambda^p -1\right)^{\frac{1}{p}}}{\lambda}\, \,  \frac{1}{ n^{\frac{1}{s}}} & \text{for $p>\mathrm{Cot}(X)$}, \\[2mm]
			E_{7}\, \dfrac{\left(\lambda^p -1\right)^{\frac{1}{p}}}{\lambda}\,  \, \frac{ 1}{n^{\frac{1}{s}-\frac{1}{t}+\frac{1}{p}}} & \text{for $p\leq  \mathrm{Cot}(X)$},\\[2mm]
			E_{8}\, \dfrac{\left(\lambda^p -1\right)^{\frac{1}{p}}}{\lambda}\,   \frac{1}{  n^{\frac{1}{p}+\frac{1}{s}-\frac{1}{2}}} & \text{for $p \leq 2$};
		\end{cases}$$
		and
		\begin{equation*}
			L_{\lambda}(B_{\ell^n_{s,t}}, p,X) \leq 	
			E_9\,\lambda\, n^{\frac{1}{\mathrm{Cot}(X)}\, - \,\frac{1}{p}+\frac{1}{s}}. 
		\end{equation*}
		\item Let $2<s \leq \infty$ and $1 \leq t \leq \infty$. If $X=\mathcal{B}(\mathcal{H})$ is finite dimensional
		\begin{equation*} 
			L_{\lambda}(B_{\ell^n_{s,t}}, p,X) \geq	\begin{cases}
				E_{1}(X)\,  \, \max \left\{\frac{1}{n^{\frac{3}{2}-\frac{1}{s}}}, \, \frac{1}{e\, n^{2\left(1-\frac{1}{s}\right)}}\right\}\, \left(\frac{\lambda^p -1}{2\lambda^p - 1}\right)^{\frac{1}{p}} & \text{for $p=1$}, \\[2mm]
				E_{3}\,\frac{1}{n^{\frac{1}{2}-\frac{1}{s}}}\, \left(\frac{\lambda^p -1}{2\lambda^p - 1}\right)^{\frac{1}{p}}\, & \text{for $p \geq 2$}, \\[2mm]
				
				E_{2}(X)\,  \max \left\{\frac{1}{n^{\frac{3}{2}-\frac{1}{s}-\theta}}, \, \frac{1}{e^{(1-\theta)}\, n^{\frac{3(1-\theta)}{2}+\frac{\theta}{s}+2}}\right\}\left(\frac{\lambda^p -1}{2\lambda^p - 1}\right)^{\frac{1}{p}}  & \text{for $1<p< 2$}, 
			\end{cases}
		\end{equation*}
		and $$L_{\lambda}(B_{\ell^n_{s,t}}, p,X) \leq d\,\lambda^{\frac{1}{\log \,n}}\,b_{[\log\,n]}(B_{\ell^n_{s,t}})\, \,\frac{1}{n^{\frac{1}{p}+\frac{1}{s}-\frac{1}{2}}}.
		$$ If $X=\mathcal{B}(\mathcal{H})$ is infinite dimensional, then $$  L_{\lambda}(B_{\ell^n_{s,t}}, p,X) \geq \begin{cases}
			E_{5}\, \dfrac{\left(\lambda^p -1\right)^{\frac{1}{p}}}{\lambda}\, \,  \frac{1}{  n^{1-\frac{1}{s}}} & \text{for $p > 2$}, \\[2mm]
			
			E_{6}\,\dfrac{\left(\lambda^p -1\right)^{\frac{1}{p}}}{\lambda}\, \,  \frac{1}{ n^{\frac{1}{s}}} & \text{for $p>\mathrm{Cot}(X)$}, \\[2mm]
			E_{7}\, \dfrac{\left(\lambda^p -1\right)^{\frac{1}{p}}}{\lambda}\,  \, \frac{ 1}{n^{\frac{1}{s}-\frac{1}{t}+\frac{1}{p}}} & \text{for $p\leq  \mathrm{Cot}(X)$},\\[2mm]
			E_{8}\, \dfrac{\left(\lambda^p -1\right)^{\frac{1}{p}}}{\lambda}\,   \frac{1}{  n^{\frac{1}{2}-\frac{1}{s}+\frac{1}{p}}} & \text{for $p \leq 2$};
		\end{cases}$$
		and
		\begin{equation*}
			L_{\lambda}(B_{\ell^n_{s,t}}, p,X) \leq 	
			E_9\,\lambda\, n^{\frac{1}{\mathrm{Cot}(X)}\, - \,\frac{1}{p}+\frac{1}{s}}. 
		\end{equation*}
	\end{enumerate}
	Here the constants $E_1, \ldots, E_9$, and $d$ are as in Theorems \ref{thm-1.2} and \ref{thm-4.1}.
\end{example}
\begin{pf}
	We divide the proof according to the geometry of the Lorentz space 
	$\ell_{s,t}$ and the corresponding embedding properties.
	
	\medskip
	\noindent
	Recall first that the dual norm of the vector $\sum_{k=1}^{n} e_k^*$ satisfies
	\[
	\Big\|\sum_{k=1}^{n} e_k^*\Big\|_{\ell_{s,t}^*}
	= n^{1-\frac{1}{s}},
	\]
	which follows from the explicit description of the dual of 
	$\ell_{s,t}$.
	
	Moreover, Lorentz spaces satisfy the lexicographical inclusion rule
	\[
	\ell_{p,q}\subseteq \ell_{s,t}
	\quad\Longleftrightarrow\quad
	p<s \;\text{ or }\; (p=s \text{ and } q\le t),
	\]
	with continuous embeddings.
	
	We now compute the norms of the relevant identity operators
	that enter the
	general estimates of Theorems~\ref{thm-1.2}, \ref{thm-1.3-a},
	and \ref{thm-1.6}.
	
	\medskip
	\noindent
	\textbf{Embedding with $\ell_\infty^n$.}
	
	For every $1\le s\le\infty$ and $1\le t\le\infty$,
	$
	\ell_{s,t}^n \subseteq \ell_\infty^n
	$,
	and hence
	$
	\|\mathrm{Id}: \ell_{s,t}^n \to \ell_\infty^n\| = 1
	$.
	Conversely,
	\[
	\|\mathrm{Id}: \ell_\infty^n \to \ell_{s,t}^n\|
	= \|(1,\dots,1)\|_{s,t}
	= n^{\frac1s}.
	\]
	
	\medskip
	\noindent
	\textbf{Embedding with $\ell_2^n$.}
	
	\medskip
	\textbf{Case $1 \le s < 2$, $1\le t\le\infty$, or $s=2$, $1\le t\le2$.}
	
	In this range we have the continuous embedding
	$\ell_{s,t}^n \subseteq \ell_2^n$.
	Hence
	$\|\mathrm{Id}: \ell_{s,t}^n \to \ell_2^n\| \le 1$.
	For the reverse embedding, using the vector $x=(1,\dots,1)$, we obtain $\|x\|_{s,t}=n^{\frac1s}$ and $\|x\|_2=n^{\frac12}$
	so that
	\[
	\|\mathrm{Id}: \ell_2^n \to \ell_{s,t}^n\|
	=\sup_{x\neq0}\frac{\|x\|_{s,t}}{\|x\|_2}
	= n^{\frac1s-\frac12}.
	\]
	
	\medskip
	\textbf{Case $2<s\le\infty$, $2\le t\le\infty$.}
	
	Here $\ell_{s,t}$ is $2$-convex (see \cite[p.~189]{defant-2003}),
	and $\ell_2^n \subseteq \ell_{s,t}^n$.
	Thus
	\[
	\|\mathrm{Id}: \ell_2^n \to \ell_{s,t}^n\| \le 1,
	\qquad
	\|\mathrm{Id}: \ell_{s,t}^n \to \ell_2^n\|
	= n^{\frac12-\frac1s}.
	\]
	
	\medskip
	\textbf{Case $2<s\le\infty$, $1\le t\le2$.}
	
	In this situation (see \cite[p.~189]{defant-2003})
	\[
	\|\mathrm{Id}: \ell_2^n \to \ell_{s,t}^n\| \le 1,
	\qquad
	\|\mathrm{Id}: \ell_{s,t}^n \to \ell_2^n\|
	\le n^{\frac12-\frac1s}.
	\]
	
	\medskip
	\noindent
	\underline{Proof of part (1).}
	
	Assume $1\le s<2$, $1\le t\le\infty$, or $s=2$, $1\le t\le2$.
	Then $\ell_{s,t}^n\subseteq \ell_2^n$ and so
	\[
	\|\mathrm{Id}: \ell_2^n \to \ell_{s,t}^n\|
	= n^{\frac1s-\frac12}.
	\]
	
	Applying Theorems~\ref{thm-1.3-a}(1) and
	\ref{thm-1.6} (1) for $q=2$, and inserting
	\[
	\|\mathrm{Id}: \ell_2^n \to \ell_{s,t}^n\|
	= n^{\frac1s-\frac12},
	\qquad
	\|\mathrm{Id}: \ell_\infty^n \to \ell_{s,t}^n\|
	= n^{\frac1s},
	\qquad
	\|\mathrm{Id}: \ell_{s,t}^n \to \ell_\infty^n\|=1,
	\]
	together with
	$\Big\|\sum_{k=1}^n e_k^*\Big\|_{\ell_{s,t}^*}
	= n^{1-\frac1s}$, 
	yields precisely the bounds stated in part $(1)$, both in the
	finite- and infinite-dimensional cases of
	$X=\mathcal{B}(\mathcal{H})$.
		
	\medskip
	\noindent
	\underline{Proof of part (2).} \noindent Let $2<s\le\infty$ and $1\le t\le\infty$.
	
	If $2\le t\le\infty$, then $\ell_{s,t}$ is $2$-convex.
	Using
	$\|\mathrm{Id}: \ell_{s,t}^n \to \ell_2^n\|
	= n^{\frac12-\frac1s},
	\qquad
	\|\mathrm{Id}: \ell_\infty^n \to \ell_{s,t}^n\|
	= n^{\frac1s}$,
	and applying Theorems~\ref{thm-1.3-a}(2) and \ref{thm-1.6} (2) for $q=2$,
	we obtain the desired bounds.
	
	If $1\le t\le2$, we use instead
	$\|\mathrm{Id}: \ell_2^n \to \ell_{s,t}^n\| \le 1$, $
	\|\mathrm{Id}: \ell_{s,t}^n \to \ell_2^n\|
	\le n^{\frac12-\frac1s}$,
	together with $\|\mathrm{Id}: \ell_\infty^n \to \ell_{s,t}^n\|
	= n^{\frac1s}$, $\|\mathrm{Id}: \ell_{s,t}^n \to \ell_\infty^n\|=1$,
	and reduce the estimates by
	Theorems~\ref{thm-1.2}, \ref{thm-1.3-a}, and \ref{thm-1.6} for $q=2$.

	\medskip
	\noindent
	The proof is complete now.
\end{pf}

\begin{example} [{\bf Orlicz spaces}]
	Let $\psi$ be an Orlicz function which satisfies the $\Delta_{2}$-condition. 
	\begin{enumerate}
		\item Let $a^2 \leq T \psi(a)$ for all $a$ and some $T>0$.  If $X=\mathcal{B}(\mathcal{H})$ is finite dimensional, then
		\begin{equation*}
			L_{\lambda}(B_{\ell^n_{\psi}}, p,X) \geq	\begin{cases}
				E_{1}(X)\, \frac{1}{\norm{Id: \ell^n_{2} \rightarrow \ell^n_{\psi}}} \, \max \left\{\frac{1}{n\, \norm{Id: \ell^n_{2} \rightarrow \ell^n_{\psi}}}, \, \frac{1}{e\,n^{\frac{3}{2}}\, \psi^{-1}(1/n)}\right\}\, \left(\frac{\lambda^p -1}{2\lambda^p - 1}\right)^{\frac{1}{p}} & \text{for $p=1$}, \\[2mm]
				E_{3}\norm{I:\ell^n _2 \rightarrow \ell^n_{\psi}}^{-\frac{2}{p}}\, \left(\frac{\lambda^p -1}{2\lambda^p - 1}\right)^{\frac{1}{p}}\, & \text{for $p \geq 2$}, \\[2mm]
				
				E_{2}(X)  \max \left\{\frac{1}{n^{1-\theta} \norm{Id: \ell^n_{2} \rightarrow \ell^n_{\psi}}^{2-\theta}},  \frac{\norm{I:\ell^n _2 \rightarrow \ell^n_{\psi}}^{-1}}{(en^{\frac{3}{2}})^{(1-\theta)} (\psi^{-1}(1/n))^{1-\theta}}\right\}\left(\frac{\lambda^p -1}{2\lambda^p - 1}\right)^{\frac{1}{p}} \hspace{-4mm}  & \text{for $1<p< 2$}, 
			\end{cases}
		\end{equation*}
		and $$L_{\lambda}(B_{\ell^n_{\psi}}, p,X) \leq d\,  \lambda^{\frac{1}{\log \,n}}\,b_{[\log\,n]}(B_{\ell^n_{\psi}})\, \,\frac{1}{n^{\frac{1}{p}}}.
		$$
		If $X=\mathcal{B}(\mathcal{H})$ is infinite dimensional, 
		$$  L_{\lambda}(B_{\ell^n_{\psi}}, p,X) \geq \begin{cases}
			E_{5}\, \dfrac{\left(\lambda^p -1\right)^{\frac{1}{p}}}{\lambda}\, \frac{1}{\norm{Id: \ell^n_{2} \rightarrow \ell^n_{\psi}}}\,  \frac{1}{  n^{\frac{1}{2}}} & \text{for $p > 2$}, \\[2mm]
			
			E_{6}\,\dfrac{\left(\lambda^p -1\right)^{\frac{1}{p}}}{\lambda}\, \,  \frac{\psi^{-1}(1/n)}{ \psi^{-1}(1)} & \text{for $p>\mathrm{Cot}(X)$}, \\[2mm]
			E_{7}\, \dfrac{\left(\lambda^p -1\right)^{\frac{1}{p}}}{\lambda}\,  \, \frac{\psi^{-1}(1/n)}{ \psi^{-1}(1)}\, \frac{ 1}{n^{\frac{1}{p}-\frac{1}{t}}} & \text{for $p\leq  \mathrm{Cot}(X)$},\\[2mm]
			E_{8}\, \dfrac{\left(\lambda^p -1\right)^{\frac{1}{p}}}{\lambda}\,   \frac{1}{\norm{Id: \ell^n_{2} \rightarrow \ell^n_{\psi}}}\, \frac{1}{  n^{\frac{1}{p}}} & \text{for $p \leq 2$};
		\end{cases}$$
		and
		\begin{equation*}
			L_{\lambda}(B_{\ell^n_{\psi}}, p,X) \leq 	
			E_9\,\lambda\, n^{\frac{1}{\mathrm{Cot}(X)}\, - \,\frac{1}{p}} \, \frac{\psi^{-1}(1)}{ \psi^{-1}(1/n)}. 
		\end{equation*}
		
		\item Let $\psi(\beta \,a) \leq T \beta^2\,\psi(a)$ for $0 \leq \beta,a \leq 1$ and some $T>0$. If $X=\mathcal{B}(\mathcal{H})$ is finite dimensional,
		\begin{equation*}
			R_{\lambda}(B_{\ell^n_{\psi}}, p,X) \geq	\begin{cases}
				E_{1}(X)\,  \frac{1}{\norm{\mathrm{Id}:\ell^n_{\psi} \to \ell^n _2  }}\, \max \left\{\frac{1}{n}, \, \frac{1}{e\,n^{\frac{3}{2}}\, \psi^{-1}(1/n)}\right\}\, \left(\frac{\lambda^p -1}{2\lambda^p - 1}\right)^{\frac{1}{p}} & \text{for $p=1$}, \\[2mm]
				E_{3}\, \frac{1}{\norm{\mathrm{Id}:\ell^n_{\psi} \to \ell^n _2  }}\, \left(\frac{\lambda^p -1}{2\lambda^p - 1}\right)^{\frac{1}{p}}\, & \text{for $p \geq 2$}, \\[2mm]
				
				\frac{E_{2}(X)}{\norm{\mathrm{Id}:\ell^n_{\psi} \to \ell^n _2  }}\,  \max \left\{\frac{1}{n^{1-\theta}}, \, \frac{1}{(en^{\frac{3}{2}})^{(1-\theta)}\, (\psi^{-1}(1/n))^{1-\theta}}\right\}\left(\frac{\lambda^p -1}{2\lambda^p - 1}\right)^{\frac{1}{p}}  & \text{for $1<p< 2$}, 
			\end{cases}
		\end{equation*}
		and 
		$$L_{\lambda}(B_{\ell^n_{\psi}}, p,X) \leq d\,  \lambda^{\frac{1}{\log \,n}}\,b_{[\log\,n]}(B_{\ell^n_{\psi}})\,\,\psi^{-1}\left(\frac{1}{n}\right)\,  n^{-\left(\frac{1}{p}-\frac{1}{2}\right)}.$$
		If $X=\mathcal{B}(\mathcal{H})$ is infinite dimensional, 
		$$  L_{\lambda}(B_{\ell^n_{\psi}}, p,X) \geq \begin{cases}
			E_{5}\, \dfrac{\left(\lambda^p -1\right)^{\frac{1}{p}}}{\lambda}\, \frac{1}{n\, \psi^{-1}(1/n)} & \text{for $p > 2$}, \\[2mm]
			
			E_{6}\,\dfrac{\left(\lambda^p -1\right)^{\frac{1}{p}}}{\lambda}\, \,  \frac{\psi^{-1}(1/n)}{ \psi^{-1}(1)} & \text{for $p>\mathrm{Cot}(X)$}, \\[2mm]
			E_{7}\, \dfrac{\left(\lambda^p -1\right)^{\frac{1}{p}}}{\lambda}\,  \, \frac{\psi^{-1}(1/n)}{ \psi^{-1}(1)}\, \frac{ 1}{n^{\frac{1}{p}-\frac{1}{t}}} & \text{for $p\leq  \mathrm{Cot}(X)$},\\[2mm]
			E_{8}\, \dfrac{\left(\lambda^p -1\right)^{\frac{1}{p}}}{\lambda}\,   \frac{1}{\psi^{-1}(1/n)}\, \frac{1}{  n^{\frac{1}{p}+\frac{1}{2}}} & \text{for $p \leq 2$};
		\end{cases}$$
		and
		\begin{equation*}
			L_{\lambda}(B_{\ell^n_{\psi}}, p,X) \leq 	
			E_9\,\lambda\, n^{\frac{1}{\mathrm{Cot}(X)}\, - \,\frac{1}{p}} \, \frac{\psi^{-1}(1)}{ \psi^{-1}(1/n)}. 
		\end{equation*}
	\end{enumerate}
	Here the constants $E_1, \ldots, E_9$, and $d$ are as in Theorems \ref{thm-1.2} and \ref{thm-4.1}.
\end{example}
\begin{pf}
	We begin by recalling two fundamental facts about finite dimensional Orlicz sequence spaces (see \cite[p.~192]{defant-2003}):
	
	\begin{equation*}\label{orlicz-basic}
		\left\|\sum_{k=1}^{n} e_k \right\|_{\ell_\psi}
		= \frac{1}{\psi^{-1}\!\left(\frac{1}{n}\right)}
		\quad \text{and} \quad
		\left\|\sum_{k=1}^{n} e_k^* \right\|_{\ell_\psi^*}
		\left\|\sum_{k=1}^{n} e_k \right\|_{\ell_\psi}
		= n .
	\end{equation*}
	
	In particular,
	\begin{equation}\label{dual-norm}
		\left\|\sum_{k=1}^{n} e_k^* \right\|_{\ell_\psi^*}
		= n \, \psi^{-1}\!\left(\frac{1}{n}\right).
	\end{equation}
On the other hand, 
\begin{equation} \label{infty-embd}
	\left\|\mathrm{Id}:\ell_\infty^n \to \ell_\psi^n\right\|
	=
	\frac{1}{\psi^{-1}\!\left(\frac{1}{n}\right)}\,\,\,\,\, \mbox{and} \,\,\,\, \left\|\mathrm{Id}:\ell_\psi^n \to \ell_\infty^n\right\|
	=
	\psi^{-1}(1).
\end{equation}

	\medskip
	
	\noindent

	\smallskip
	
	\textbf{Case (1).}
	Assume that
	$	a^2 \le T \psi(a)$ for all  $ a \ge 0 
	$.
	This condition implies the continuous embedding
	$
	\ell_\psi \hookrightarrow \ell_2 
	$.
	Thus, all estimates in Theorems~\ref{thm-1.3-a} (1) and \ref{thm-1.6} (1)
	apply to $\ell_\psi^n$ with the facts \eqref{dual-norm} and \eqref{infty-embd}.
		
	\medskip
	
	\textbf{Case (2).}
	Assume that
	$
	\psi(\beta a) \le T \beta^2 \psi(a)$,	$0 \le \beta,a \le 1$.
	This is precisely the condition ensuring that $\ell_\psi$ is
	$2$-convex (see \cite[p.189]{defant-2003}). Therefore
	the results of Theorems \ref{thm-1.3-a} (2) and \ref{thm-1.6} (2)
	can be applied with the facts \eqref{dual-norm} and \eqref{infty-embd}.
	
	\medskip
	Combining these completes the proof.
\end{pf}
\begin{rem}
	By arguments analogous to those used in the proofs of Theorems  \ref{thm-1.2}, \ref{thm-4.1}, \ref{thm-1.3-a}, and \ref{thm-1.6}, one can study the asymptotic behavior of the second Bohr radius constant $B_{\lambda}(B_{Z}, p,X)$ for holomorphic functions and show that the asymptotic behaviors of $L_{\lambda}(B_{Z}, p,X)$ and $B_{\lambda}(B_{Z}, p,X))$ coincide up to a multiplicative constant. These results can likewise be applied to Banach sequence spaces, including the standard examples discussed in Section $5$, leading to the same conclusion: the asymptotic behaviors of $L_{\lambda}(B_{Z}, p,X)$ and $B_{\lambda}(B_{Z}, p,X))$ agree up to a scalar multiple. We therefore omit the proof, as it follows along similar lines. 
\end{rem}
\noindent{\bf Statements and Declarations:}\\

\noindent{\bf Acknowledgment.} The author sincerely acknowledges Professor B. V. Rajarama Bhat, ISI Bangalore, for support through the J. C. Bose Fellowship during the final stages of this manuscript. The author is currently supported by an NBHM Postdoctoral Fellowship from the Department of Atomic Energy (DAE), Government of India.\\

\noindent{\bf Conflict of interest.} The author declares that there is no conflict of interest regarding the publication of this paper.\\

\noindent{\bf Data availability statement.} Data sharing not applicable to this article as no datasets were generated or analysed during the current study.\\

\noindent{\bf Competing Interests.} The author declares none.\\

\noindent{\bf ORCID Id.} 0000-0001-9466-
178X

\end{document}